\pdfoutput=1
\documentclass[a4paper,reqno]{amsart}

\usepackage[utf8]{inputenc}
\usepackage[T1]{fontenc}
\usepackage{newtxtext}

\usepackage{braket}
\usepackage{epigraph}
\usepackage{microtype}
\usepackage{hyphenat}
\usepackage{amsmath,amssymb,amsfonts,amsthm}
\usepackage{mathtools}
\usepackage{mathrsfs}
\usepackage{dsfont}
\usepackage{siunitx}
\usepackage{cancel}

\usepackage[dvipsnames]{xcolor}
\usepackage{graphicx}
\usepackage[all,cmtip]{xy} 
\usepackage{xcolor}
\usepackage[normalem]{ulem}

\usepackage{tikz}
\usepackage{tikz-cd}
\usepackage{pgfplots}
\usepackage{tikz-3dplot}
\usetikzlibrary{calc, shadings, decorations.markings, shapes, arrows}
\pgfplotsset{compat=1.18}
\usepackage{wrapfig}
\usepackage{subfiles}

\usepackage{booktabs}
\usepackage{tabularx}
\usepackage{enumitem}
\usepackage{longtable}
\usepackage{rotating}

\usepackage{xspace}
\usepackage[printonlyused,withpage]{acronym}
\usepackage{lipsum} 
\usepackage{listings}

\usepackage[square, numbers]{natbib}
\citeindextrue

\usepackage[colorlinks=true, linkcolor=blue, citecolor=blue, urlcolor=blue]{hyperref}

\definecolor{light-gray}{gray}{0.95}
\usepackage[color=light-gray,textwidth=2cm]{todonotes}

\def\XXint#1#2#3{{\setbox0=\hbox{$#1{#2#3}{\int}$} 
		\vcenter{\hbox{$#2#3$}}\kern-.5\wd0}}

\renewcommand{\theta}{\vartheta}
\renewcommand{\epsilon}{\varepsilon}

\makeatletter
\renewcommand{\fnum@figure}{\textsc{\figurename~\thefigure}} 
\makeatother

\def\Proofsketch{{\medbreak\noindent{\textit{Sketch of proof.} }}}
\def\endproof{~\hfill $\blacksquare$\par\bigskip}

\newtheoremstyle{classicthm}
{12pt}{12pt}{\slshape}{}{\bfseries}{.}{.5em}{}

\theoremstyle{classicthm}
\newtheorem{theoremd}{Theorem}[section]
\newenvironment{theorem}{\begin{theoremd}}{ \end{theoremd}}

\newtheorem{theoremd*}{Theorem}
\newenvironment{theorem*}{\begin{theoremd*}}{ \end{theoremd*}}

\newtheorem{cord}[theoremd]{Corollary}

\newtheorem{lemd}[theoremd]{Lemma}

\newtheorem{propd}[theoremd]{Proposition}
\newenvironment{prop}{\begin{propd}}{ \end{propd}}

\newtheorem{definitiond}[theoremd]{Definition}
\newenvironment{definition}{\begin{definitiond}}{ \end{definitiond}}

\newtheoremstyle{remarkstyle}
{12pt}{12pt}{\upshape}{}{\itshape}{.}{.5em}{} 

\theoremstyle{remarkstyle}

\newtheorem{remd}[theoremd]{Remark}
\newenvironment{rem}{\begin{remd}}{	\end{remd}}

\newtheorem{exampled}[theoremd]{Example}

\newtheorem{ossd}[theoremd]{Observation}

\newtheorem{noted}{Note}[theoremd]

\title[THE FRIED CONJECTURE FOR MORSE--SMALE FLOWS]{THE FRIED CONJECTURE FOR MORSE--SMALE FLOWS: A SURVEY ON RAY--SINGER AND MILNOR METRICS}
\author{Giovanni Molinari}
\address{Department of Applied Mathematics, University of Waterloo, 200 University Avenue West, Waterloo, Ontario, Canada, N2L 3G1}
\address{Perimeter Institute for Theoretical Physics, 31 Caroline Street North, Waterloo, Ontario, Canada N2L 2Y5}
\email{gmolinar@uwaterloo.ca} 
\date{\today} 
\begin{document}
\begin{abstract}
This paper reviews the generalised Fried conjecture for Morse--Smale flows on compact Riemannian manifolds. We first establish the proper framework by constructing the twisted de Rham complex and deriving the Hodge decomposition, which underpins the definition of the Ray--Singer torsion. On the dynamical side, we characterise Morse--Smale vector fields, introducing the Ruelle Zeta function to encode the spectral data of closed orbits and {reviewing the Thom--Smale complex for purely gradient flows. In general, the local contributions of the critical elements of the flow are globally assembled via a Smale filtration to construct the Milnor metric} on the determinant line of the twisted cohomology. {Finally, under the assumption of a unitary representation, we present the theorem proving the conjecture: the Ray--Singer metric coincides with the Milnor metric. Furthermore, we discuss the structure of the Milnor metric, which factors into the Ruelle Zeta function evaluated at zero and a specific contribution from the fixed points. We clarify that this object reduces to either the classical algebraic Thom--Smale torsion or the Ruelle Zeta function solely in the limiting cases of purely gradient or fixed-point-free flows.}
\end{abstract}
\maketitle
\tableofcontents
\section{Introduction}

The relation between the topological structure of a manifold and the dynamical properties of flows defined on it is a central theme in modern mathematics. Well-known manifestations of this link include the Poincaré--Hopf theorem, which relates the indices of vector field singularities to the Euler characteristic, and Morse theory, which reconstructs the homology of a manifold from the critical points of a gradient flow \cite{Milnor63, Smale61}. However, to probe the deeper geometric properties of these spaces, one requires invariants capable of detecting fine structural nuances beyond standard topological data. Among these, torsion occupies a privileged position.

Historically, the concept of torsion emerged in combinatorics during the 1930s with the work of Reidemeister \cite{Reidemeister35} and Franz \cite{Franz35}, who introduced it to classify lens spaces indistinguishable by homology and homotopy groups alone.For decades, this combinatorial torsion remained a purely algebraic object, distinct from the analytic framework of differential geometry. This separation was bridged in the 1970s by Ray and Singer \cite{RaySinger71}, who introduced the analytic torsion. Defined via the spectral properties of the Laplacian acting on differential forms, this quantity depends \textit{a priori} on the choice of a Riemannian metric. However, the celebrated Cheeger-Müller theorem \cite{Cheeger79, Muller78} subsequently established the equality of the analytic torsion with the combinatorial one for compact manifolds. A crucial consequence of this result is the proof that the analytic torsion is, in fact, independent of the auxiliary Riemannian metric used in its construction.

With the equivalence between the combinatorial and analytic perspectives established, the focus shifted towards a dynamical interpretation of these invariants. In the 1980s, David Fried \cite{Fried87} proposed a visionary conjecture linking the spectral geometry of a manifold to the recurrence properties of a flow. He posited that the analytic torsion could be entirely recovered from the minimal periods of closed orbits, which are rigorously encoded in the Ruelle Zeta function. Specifically, Fried's insight states that the value of this function at zero should coincide exactly with the analytic torsion, implying that the Laplacian spectrum is determined by the lengths of the system's periodic trajectories.

This paper is dedicated to reviewing a rigorous derivation of this equivalence within the specific setting of Morse--Smale flows. Unlike chaotic systems such as Anosov flows, where the number of periodic orbits is generally infinite, Morse--Smale flows represent the simplest class of structurally stable dynamical systems. Their essential dynamics are captured by the non-wandering set -- the locus of points exhibiting recurrent behaviour -- which in this tame setting is strictly confined to a finite number of hyperbolic fixed points and closed orbits. These isolated critical elements organise the global flow via their invariant manifolds. Heuristically, the stable manifold of a critical element consists of all trajectories that are asymptotically attracted to it in the distant future, while the unstable manifold collects the trajectories that originate from it in the remote past. The defining requirement of Morse--Smale flows is that these incoming and outgoing surfaces intersect as cleanly as possible, a condition known as transversality.

A distinctive feature of our framework is the use of differential forms taking values in a flat vector bundle $\mathcal{E}$, associated with a unitary representation of the fundamental group of the base manifold $M$. The introduction of a non-trivial representation was primarily motivated by the requirement of acyclicity. By ensuring the vanishing of the cohomology, the invariants involved could be treated as scalar values, thereby simplifying the analysis.
However, in line with modern approaches, we adopt the bundle formalism not merely to dispense with the acyclicity hypothesis, but because it constitutes the most general and natural geometric framework in which to phrase the problem.

Consequently, a significant portion of our work is devoted to establishing the foundations of this twisted calculus, including the construction of the twisted de Rham differential and the associated Hodge theory. This cohomological framework is essential to understand the modern extension of the Fried conjecture. Indeed, when the cohomology does not vanish, the original scalar equality posited by Fried must be upgraded to an equivalence between two specific metrics constructed on the determinant line of the twisted cohomology: the Ray--Singer metric and the Milnor metric.

{On the dynamical side, the Milnor metric is constructed by means of a Smale filtration, which decomposes the manifold into submanifolds isolating the individual critical elements. Although this metric factors into the Ruelle Zeta function evaluated at zero and a second term, the latter cannot generally be interpreted as a contribution arising solely from the fixed points, since heteroclinic orbits inherently couple the two dynamical sectors. It is only in the restricted setting of purely gradient flows that this term explicitly coincides with the algebraic torsion of the Thom--Smale complex. Conversely, for fixed-point-free flows, the Milnor metric strictly reduces to the Ruelle Zeta function.}

Within this framework, the unitarity assumption for the representation is maintained as a convenient technical hypothesis: it guarantees that the two metrics under comparison coincide exactly, without the need for additional correction terms. 

Drawing on the seminal works of Fried \cite{Fried87} and the recent generalisations by Shen and Yu \cite{Shen21} for arbitrary Morse--Smale flows, we provide a unified proof of this metric equivalence, which serves as the formal validation of Fried's intuition, demonstrating that the spectral properties of the manifold are holographically encoded in the structure of the underlying flow.

\vspace{\baselineskip}

In Section \ref{sec: twisted-differential}, we establish the geometric framework by constructing the flat vector bundle $\mathcal{E}$ via a unitary representation and defining the twisted de Rham complex. We also introduce the essential differential operators, such as the twisted de Rham differential, its formal adjoint operator and the twisted Laplacian.

In Section \ref{sec: Hodge-decomposition}, we derive the Hodge decomposition for twisted differential forms and establish the isomorphism between harmonic forms and the twisted de Rham cohomology.

In Section \ref{sec: M-S-flow}, we characterise Morse--Smale flows, defining their non-wandering set in terms of hyperbolic fixed points and closed orbits, and describing the geometry of their stable and unstable manifolds.

In Section \ref{sec: Thom--Smale-torsion}, we review the general theory of algebraic torsion for finite-dimensional complexes and {apply it to construct the Thom--Smale complex for the restricted class of purely gradient flows, underlying how heteroclinic orbits prevent this construction in the general setting.}

In Section \ref{sec: analytic-torsion-zeta}, we define the Ray--Singer analytic torsion through the zeta-regularisation of the Laplacian and introduce the Ruelle Zeta function to capture the spectral data of closed orbits. 

Finally, in Section \ref{sec: Fried-conjecture}, we construct the Milnor metric on the determinant line of the twisted cohomology {via a Smale filtration. We then present the generalised Fried conjecture, establishing the exact equality between the Ray--Singer and Milnor metrics under the unitary assumption, and we explore how the Milnor metric factorises into known invariants strictly in decoupled dynamical regimes.}

\section*{Acknowledgements}
This paper collects part of the preliminary work carried out during the author's Master Thesis at the University of Pavia, and it is written with the aim of offering a concise account of the current state of the art on the subject, intended to support a companion research article soon to appear. I would like to thank M.\ Schiavina and C.\ Dappiaggi for their guidance and support, and for many valuable comments and discussions throughout the preparation of this work.

\section{Twisted de Rham Differential}
\label{sec: twisted-differential}

In this section we extend the standard notion of the de Rham differential to smooth forms taking values in a suitable vector bundle. For the construction of the twisted differential, we mainly refer to \cite{MathaiWu11Twisted} and \cite{DangRiviere17Anisotropic}, while we rely on \cite{Lee12Smooth, Hatcher02, Munkres00} for the general mathematical background.

\begin{proof}[Conventions and geometric setup]
\label{conventions-setup}
Throughout this paper, unless stated otherwise, we shall assume the following mathematical setting. Let $(M, g)$ be a compact, connected, oriented Riemannian manifold of dimension $n$ with empty boundary ($\partial M = \emptyset$). We denote by $\mathfrak{X}(M)$ the space of smooth vector fields on $M$ and to each $V \in \mathfrak{X}(M)$, we assign the associated flow map $\phi^t_V: M \to M$. It is important to note that due to the compactness of $M$, this flow is complete, meaning it is defined for all $t \in \mathbb{R}$. 

Furthermore, given $E$ a complex vector space of dimension $N$, we fix a {unitary representation $\rho: \pi_1(M) \to \mathrm{GL}(E)$.\footnote{{While we assume $\rho$ to be unitary, the main results presented in this survey actually hold under the slightly less restrictive assumption that the representation is unimodular, \textit{i.e.}, $|\det \rho([\gamma])| = 1$ for all $[\gamma] \in \pi_1(M)$.}} Since $M$ is connected, we omit the reference to the base point in the notation of the fundamental group $\pi_1(M)$.}
\end{proof}

\begin{definition}
\label{total-space}
    Given $\rho: \pi_1(M) \to \mathrm{GL}(E)$, we define the total space $\mathcal{E}$ as:
    \begin{equation}
    \label{eq:equivalence-relation}
        \mathcal{E} \coloneqq (E \times \widetilde{M})/ \sim,
    \end{equation}
    where $\widetilde{M}$ stands for the universal cover of $M$ and the equivalence relation is defined by the action of the fundamental group via deck transformations:
    \begin{equation}
        (v, x\cdot [\gamma]) \sim (\rho([\gamma])v, x), \quad \forall [\gamma] \in \pi_1(M), \quad \forall x \in \widetilde{M}, \quad \forall v \in E.
    \end{equation}
    The space $\mathcal{E}$ is endowed with a natural projection $\mathbb{P}: \mathcal{E} \to M$, given by $\mathbb{P}([v, x]) \coloneqq \pi(x)$, where $\pi: \widetilde{M} \to M$ is the covering map.
\end{definition}

It is natural to ask whether the constructed space $\mathcal{E}$ works as the total space of a smooth vector bundle over $M$. The following proposition provides the answer.

\begin{prop}
\label{vector-bundle}
    Let $\mathcal{E}$ and $\mathbb{P}: \mathcal{E} \to M$ be as per Definition \ref{total-space}. Then, the triple $(M, \mathcal{E}, \mathbb{P})$ is a smooth vector bundle of rank $N$.
\end{prop}

\Proofsketch
The result follows by verifying the hypotheses of the \textit{Vector Bundle Chart Lemma} \cite[Lemma 10.6]{Lee12Smooth}. Let $\{U_\alpha\}$ be an open cover of $M$ consisting of simply connected sets with connected intersections. For each $\alpha$, we fix a smooth local section $\widetilde{s_\alpha}: U_\alpha \to \widetilde{M}$ of the universal covering map $\pi$. We define the local trivialisation $\Phi_\alpha: \mathbb{P}^{-1}(U_\alpha) \to U_\alpha \times E$ as
\begin{equation*}
    [\widetilde{q}, v] \mapsto (q, \rho([\gamma]_q)v),    
\end{equation*}
where $q=\pi(\widetilde{q})$ and $[\gamma]_q \in \pi_1(M)$ is the unique deck transformation satisfying $\widetilde{q} = \widetilde{s_\alpha}(q) \cdot [\gamma]_q$.\\
One can verify that $\Phi_\alpha$ is well-defined, bijective, and restricts to a linear isomorphism on each fibre. The smoothness relies on the transition maps. On an overlap $U_\alpha \cap U_\beta$, the chosen sections differ by a deck transformation $[\gamma_{\alpha\beta}]_q$ such that $\widetilde{s_\alpha}(q) = \widetilde{s_\beta}(q) \cdot [\gamma_{\alpha\beta}]_q$. Since the domain is connected and $\pi_1(M)$ is discrete, $[\gamma_{\alpha\beta}]_q$ must be a constant. A direct computation shows that the transition is given by:
\[
    (\Phi_\beta \circ \Phi_\alpha^{-1})(q, v) = (q, \rho([\gamma_{\alpha\beta}])v).
\]
Consequently, the transition functions $g_{\alpha\beta}(q) = \rho([\gamma_{\alpha\beta}])$ are locally constant and therefore smooth.
\endproof

It is interesting to characterise the smooth sections of this vector bundle.

\begin{rem}
\label{sections}
There exists a one-to-one correspondence between smooth sections $s \in \Gamma(M, \mathcal{E})$ and smooth equivariant maps $\phi: \widetilde{M} \to E$, \textit{i.e.}, maps that satisfy:
\begin{equation}
    \phi(x \cdot [\gamma]) = \rho([\gamma])^{-1}\phi(x), \quad \forall [\gamma] \in \pi_1(M), \quad \forall x \in \widetilde{M}.
\end{equation}
Given such a map, the corresponding section is defined by $s(q) = [\phi(x), x]$ for any lift $x \in \pi^{-1}(q)$.
\end{rem}

Let us now introduce a calculus on $\mathcal{E}$. We denote by $\Omega^{\bullet}(M, \mathcal{E})$ the space of $\mathcal{E}$-valued differential forms. While the standard wedge product is not defined internally, the space admits a module structure over scalar forms via a map $\wedge: \Omega^k(M) \times \Omega^l(M, \mathcal{E}) \to \Omega^{k+l}(M, \mathcal{E})$. To define a differential operator on them, however, we require a connection.

\begin{theorem}
\label{E-differential}
Let $\nabla: \mathfrak{X}(M) \times \Gamma(M, \mathcal{E}) \to \Gamma(M, \mathcal{E})$ be a connection. There exists a unique linear operator $d^\nabla: \Omega^\bullet(M, \mathcal{E}) \to \Omega^{\bullet + 1}(M, \mathcal{E})$, called the exterior covariant derivative, which satisfies the following properties:
\begin{enumerate}[label=(\Roman*)]
    \item \textbf{Leibniz Rule:} $d^\nabla (\alpha \wedge \omega) = d\alpha \wedge \omega + (-1)^{\deg \alpha} \alpha \wedge d^\nabla \omega, \quad \forall \alpha \in \Omega^\bullet(M), \;\forall \omega \in \Omega^\bullet(M, \mathcal{E})$;
    \item \textbf{Action on sections:} $d^\nabla s = \nabla s, \quad  \forall s \in \Gamma(M, \mathcal{E})$;
    \item The square of the operator is related to the curvature 2-form $F_\nabla$ of the connection by:
    \begin{equation}
        d^\nabla(d^\nabla \omega) = F_\nabla \wedge \omega, \quad \forall \omega \in \Omega^\bullet(M, \mathcal{E}).
    \end{equation}
\end{enumerate}
For the proof of existence, uniqueness, and the curvature formula, we refer to \cite[Theorem 12.27 and Proposition 12.64]{Lee09ManifoldsDiffGeom}.
\end{theorem}

Unlike the standard de Rham differential, $d^\nabla$ is not necessarily nilpotent, as shown by Item \textit{(III)}. To obtain a cochain complex, we must restrict our attention to flat connections.

\begin{definition}
\label{twisted-differential}
A connection $\nabla$ is called {flat} if its curvature vanishes, i.e., $F_\nabla \equiv 0$. In this case $d^\nabla$ is referred to as the {twisted de Rham differential} and the pair $(\Omega^{\bullet}(M, \mathcal{E}), d^{\nabla})$ forms the {twisted de Rham complex}.
\end{definition}

To establish an $L^2$-structure on the space of twisted forms, we equip the vector bundle $\mathcal{E}$ with a Hermitian metric $h$. Any such metric enables us to define a scalar-valued wedge product and a generalised Hodge dual. However, the assumption that $\rho$ is a unitary representation allows us to choose a specific Hermitian metric $h$ that is compatible with the flat connection $\nabla$ (\textit{i.e.}, $\nabla h = 0$). As will be discussed in Section \ref{sec: Fried-conjecture}, this compatibility guarantees that the Milnor and Ray--Singer metrics coincide exactly, thereby preventing the appearance of anomaly correction terms.

\begin{definition}
\label{scalar-wedge-product}
Let $h$ be a Hermitian metric on the vector bundle $\mathcal{E}$. We define the pairing $\wedge_h: \Omega^i(M, \mathcal{E}) \times \Omega^j(M, \mathcal{E}) \to \Omega^{i+j}(M)$ as the unique bilinear extension of the operation defined on decomposable forms by:
\begin{equation}
    (\alpha \otimes s) \wedge_h (\beta \otimes t) \coloneqq h(s, t) \, (\alpha \wedge \beta),
\end{equation}
where $\alpha \in \Omega^i(M)$, $\beta \in \Omega^j(M)$ and $s, t \in \Gamma(M, \mathcal{E})$. By definition, this global pairing satisfies the graded commutativity relation $\omega \wedge_h \eta = (-1)^{ij}{\eta \wedge_h \omega}$.
\end{definition}

Moreover, the Riemannian metric $g$ on $M$ and the Hermitian metric $h$ on $\mathcal{E}$ combine to define the twisted Hodge star.

\begin{definition}
\label{Hodge-twisted}
The \textit{twisted Hodge dual} is the unique isomorphism $*_h : \Omega^{k}(M, \mathcal{E}) \to \Omega^{n-k}(M, \mathcal{E})$ satisfying:
\begin{equation}
    \omega \wedge_h *_h \eta = \langle \omega, \eta \rangle_{gh} \, \mu, \quad \forall \omega, \eta \in \Omega^k(M, \mathcal{E}),
\end{equation}
where $\langle \cdot, \cdot \rangle_{gh}$ is the pointwise inner product and $\mu$ is the Riemannian volume form. The operator satisfies the identity $*_h *_h = (-1)^{k(n-k)} \mathbb{I}$ on $k$-forms.
\end{definition}

This leads to the definition of a global inner product, turning the space of twisted forms into a pre-Hilbert space.

\begin{definition}
\label{scalar-product-forms}
The $L^2$-inner product on $\Omega^k(M, \mathcal{E})$ is defined by:
\begin{equation}
    (\omega, \eta) \coloneqq \int_M \omega \wedge_h *_h \eta, \quad \forall \omega, \eta \in \Omega^{k}(M, \mathcal{E}).
\end{equation} 
We denote by $L^2(M, \Lambda^\bullet T^*M \otimes \mathcal{E})$ the Hilbert space completion of $\Omega^\bullet(M, \mathcal{E})$.
\end{definition}

We now have sufficient structure to provide the definition of twisted de Rham cohomology.

\begin{definition}
\label{twisted-cohomology}
The $k$-th twisted de Rham cohomology group for a flat vector bundle with connection $\nabla: \mathfrak{X}(M) \times \Gamma(M, \mathcal{E}) \to \Gamma(M, \mathcal{E})$ is defined as the quotient:
\begin{equation}
H^{{k}}(M,\mathcal{E}) \coloneqq
\frac{\operatorname{Ker} \left(d^{\nabla}_{{k}} : \Omega^{{k}}(M, \mathcal{E}) \to \Omega^{{k+1}}(M, \mathcal{E})\right)}
{\operatorname{Im} \left(d^{\nabla}_{{k-1}} : \Omega^{{k-1}}(M, \mathcal{E}) \to \Omega^{{k}}(M, \mathcal{E})\right)}.
\end{equation}
The twisted Betti numbers are defined as $b_{{k}} \coloneqq \operatorname{dim}(H^{{k}}(M, \mathcal{E}))$.
\end{definition}

To conclude this section, we introduce the operators required for the Hodge decomposition: the codifferential and the Laplacian.

\begin{definition}
\label{def: codiff-laplacian}
The codifferential $d^{\nabla, \dagger}: \Omega^{k+1}(M, \mathcal{E}) \to \Omega^k(M, \mathcal{E})$ is defined as the formal adjoint of $d^\nabla$ with respect to the $L^2$-product:
\begin{equation}
    (d^\nabla \omega, \eta) = (\omega, d^{\nabla, \dagger} \eta), \quad \forall \omega \in \Omega^k(M, \mathcal{E}), \; \forall \eta \in \Omega^{k+1}(M, \mathcal{E}).
\end{equation}
The {twisted Laplacian} $\Delta : \Omega^k(M, \mathcal{E}) \to \Omega^k(M, \mathcal{E})$ is defined as:
\begin{equation}
   \Delta \coloneqq d^{\nabla, \dagger} d^\nabla + d^\nabla d^{\nabla, \dagger}.
\end{equation}
\end{definition}

\begin{rem}
\label{co-adjoint}
There exists a useful relation between the differential and its adjoint. Using the properties of the Hodge star and integrating by parts, the codifferential can be explicitly expressed as:
\begin{equation}
    d^{\nabla, \dagger}_k = (-1)^{nk+1} *_h d^\nabla_{n-k-1} *_h,
\end{equation}
where $n = \operatorname{dim}(M)$ and $*_h$ is the twisted Hodge operator as per Definition \ref{Hodge-twisted}.
\end{rem}

\section{Hodge Decomposition}
\label{sec: Hodge-decomposition}

The Hodge decomposition provides a fundamental splitting of the space of differential forms, intimately linking the analytic properties of the Laplacian to the topological structure of the manifold. For the rigorous proofs of the subsequent results, we refer the reader to \cite[Chapter 6]{Warner1983}. It is worth noting that while this reference deals with standard scalar-valued forms, the results extend naturally to the twisted setting.

Let us begin by recalling that the Laplacian $\Delta$ is self-adjoint with respect to the $L^2$-inner product in Definition \ref{scalar-product-forms}. Its kernel defines the space of harmonic forms:

\begin{equation}
\label{eq:harmonic}
    \mathcal{H}^{k}(M, \mathcal{E}) \coloneqq \{\omega \in \Omega^{k}(M, \mathcal{E}) : \Delta\omega = 0\}.
\end{equation}

\begin{prop}
\label{harmonic-exact}
    A twisted form $\alpha \in \Omega^k(M, \mathcal{E})$ is harmonic if and only if it is closed and coclosed, \textit{i.e.},
    \[ \Delta\alpha=0 \iff d^\nabla \alpha=0 \quad \text{and} \quad d^{\nabla,\dagger} \alpha=0. \]
\end{prop}
 
Since the operator $\Delta$ is elliptic and self-adjoint on a compact manifold, the following theorem holds.

\begin{theorem}[\textbf{Hodge Decomposition}]
\label{thm:hodge-decomposition}
 The space of twisted $k$-forms admits the following orthogonal direct sum decomposition:
\begin{equation}
\label{eq:Hodge-dec}
    \Omega^k(M, \mathcal{E}) = d^\nabla(\Omega^{k-1}(M, \mathcal{E})) \oplus d^{\nabla,\dagger}(\Omega^{k+1}(M, \mathcal{E})) \oplus \mathcal{H}^{k}(M, \mathcal{E}).
\end{equation}
Consequently, every form $\omega$ can be uniquely written as $\omega = d^\nabla \alpha + d^{\nabla, \dagger} \beta + \gamma$, where $\gamma$ is harmonic.
\end{theorem}

This analytical decomposition has immediate topological consequences. Specifically, each cohomology class contains a unique harmonic representative.

\begin{theorem}
\label{thm: hamonic-cohomology-iso}
 There exists a natural isomorphism between the space of harmonic forms as per Equation \eqref{eq:harmonic} and the twisted de Rham cohomology as per Definition \ref{twisted-cohomology}:
 \[ \mathcal{H}^k(M, \mathcal{E}) \cong H^k(M, \mathcal{E}). \]
\end{theorem}

We provide a sketch of the proof to illustrate how the orthogonality property of the Hodge decomposition is essential in establishing the isomorphism.

\Proofsketch
Let $\omega$ be a closed form. Using the decomposition in Equation \eqref{eq:Hodge-dec}, we can write $\omega = d^\nabla \alpha + d^{\nabla, \dagger} \beta + \gamma$. Since $d^\nabla \omega = 0$, applying $d^\nabla$ and Proposition \ref{harmonic-exact} yields $d^\nabla d^{\nabla, \dagger} \beta = 0$. Taking the inner product with $\beta$, we find:
\[ (d^\nabla d^{\nabla, \dagger} \beta, \beta) = (d^{\nabla, \dagger} \beta, d^{\nabla, \dagger} \beta) = \|d^{\nabla, \dagger} \beta\|^2 = 0, \]
which implies $d^{\nabla, \dagger} \beta = 0$. Thus, $\omega = d^\nabla \alpha + \gamma$, meaning that $\omega$ is in the same cohomology class of $\gamma$.
To prove uniqueness, suppose $\gamma_1, \gamma_2 \in \mathcal{H}^k(M, \mathcal{E})$ represent the same class. Then $\gamma_1 - \gamma_2 = d^\nabla \eta$. Since the difference of harmonic forms is harmonic, $\gamma_1 - \gamma_2$ is both exact and harmonic (thus orthogonal to all exact forms), which implies $\|\gamma_1 - \gamma_2\|^2 = (\gamma_1 - \gamma_2, d^\nabla \eta) = (d^{\nabla, \dagger}(\gamma_1 - \gamma_2), \eta) = 0$. Hence $\gamma_1 = \gamma_2$.
\endproof

\begin{rem}
\label{rem: finite-dimension-harmonic-form}
It is a standard topological result that the cohomology groups of a compact manifold are finite-dimensional, see \cite[Theorem 10.17]{Lee09ManifoldsDiffGeom} for a detailed proof. Consequently, the isomorphism established in Theorem \ref{thm: hamonic-cohomology-iso} implies that the space of harmonic forms $\mathcal{H}^k(M, \mathcal{E})$ is also finite-dimensional.
\end{rem}

\section{Morse--Smale Vector Fields}
\label{sec: M-S-flow}

In this section, we introduce the notion of Morse--Smale vector fields and their associated flows. Since Morse--Smale dynamics constitute one of the central protagonists of the Fried conjecture and lie at the very core of this review, we deem it appropriate to provide a detailed overview of their formal definition and fundamental structural properties. For a more comprehensive treatment, we refer the reader to \cite{DangRiviere17Topology}, \cite[Appendix A]{DangRiviere17Anisotropic}, and \cite{Shen21Survey}.

The first key concept to introduce is the non-wandering set, which heuristically captures the large-time behaviour and recurrence of the system, as it contains points arbitrarily close to which trajectories return at large-enough time.

\begin{definition}
\label{nw}
   The non-wandering set is defined as:
   \begin{equation}
   \mathbf{NW}(V) \coloneqq \left \{p \in M : \forall U \ni p, \, \forall T > 0, \quad U \cap \bigcup_{t \geq T} \phi^t(U) \neq \emptyset \right \}.
\end{equation}
It is a closed, flow-invariant subset of $M$.
\end{definition}

Morse--Smale flows are characterised by the fact that this recurrent behaviour is confined strictly to fixed points and closed orbits, and that these elements are hyperbolic, in the sense defined below.

\begin{definition}
\label{hyperbolic-point}
    A fixed point $p \in M$ (where $V_p = 0$) is \textit{hyperbolic} if the differential of the flow $d_{p}\phi^t : T_{p}M \to T_{p}M$ has no eigenvalues of modulus $1$ for all $t > 0$.
\end{definition}

For a hyperbolic fixed point $p$, the tangent space splits into $d\phi^t$-invariant subspaces, $T_{p}M = T^s_{p} \oplus T^u_{p}$, corresponding to eigenvalues with modulus less than and greater than 1, respectively (see \cite[Appendix A.2]{DangRiviere17Anisotropic}). This splitting implies precise exponential estimates: there exist constants $C \ge 0$ and $\theta > 0$ such that for all $t \ge 0$,
\begin{equation}
    \left\|d_p\phi^{t}(v_s)\right\| \le C e^{-\theta t} \left\|v_s\right\| \quad \text{and} \quad \left\|d_p\phi^{-t}(v_u)\right\| \le C e^{-\theta t} \left\|v_u\right\|, 
\end{equation}
for any $v_s \in T^s_p$ and $v_u \in T^u_p$. Associated with such a point, we define the stable and unstable manifolds as the sets of points whose trajectories asymptotically approach $p$ in the future or past, respectively:
\begin{equation}
\label{eq: unstable-point}
W^s(p) \coloneqq \left\{ q \in M : \lim_{t\to+\infty} \phi^t(q) =p \right\}, \quad W^u(p) \coloneqq \left\{ q \in M : \lim_{t\to-\infty} \phi^t(q) =p \right\}.
\end{equation}
A fundamental result links the linear geometry to these manifolds, stating that $T_{p}W^s(p)=T^s_{p}$ and $T_{p}W^u(p) = T^u_{p}$, see \cite{Perko2001}.

Similar things can be said about closed orbits of the flow.

\begin{definition}
\label{hyperbolic-orbit-poincare}
A closed orbit $\Lambda$ of period $T_\Lambda > 0$ is \textit{hyperbolic} if for any $p \in \Lambda$, the map $d_{p}\phi^{T_\Lambda}$ has $1$ as a simple eigenvalue (associated with $V_p$), with all other eigenvalues having modulus different from $1$.
\end{definition}

The hyperbolicity of a closed orbit $\Lambda$ implies a continuous, flow-invariant splitting of the tangent bundle restricted to the orbit:
\begin{equation}
    \label{eq:splitting_orbit}
    TM|_\Lambda = \mathbb{R}V \oplus T^s_\Lambda \oplus T^u_\Lambda,
\end{equation}
where $T^s_\Lambda$ and $T^u_\Lambda$ correspond to the contracting and expanding directions normal to the flow. As with fixed points, this leads to exponential decay estimates. There exist $C \ge 0$ and $\theta > 0$ such that for all $p \in \Lambda$, $t \ge 0$:
\begin{equation}
   \left\|d_p\phi^{t}(v_s)\right\| \le Ce^{-\theta t} \left\|v_s\right\| \quad \text{and} \quad  \left\|d_p\phi^{-t}(v_u)\right\| \le Ce^{-\theta t} \left\|v_u\right\|,
\end{equation}
whenever $v_s \in T_p^s$ and $v_u \in T_{p}^u$. The stable and unstable manifolds for the orbit are defined as:
\begin{equation}
\label{eq:unstable-man-orbit}
\begin{split}
    W^s(\Lambda) &\coloneqq  \left\{ q \in M : \lim_{t\to+\infty} d_M(\phi^t(q), \Lambda) = 0 \right\}, \\
    W^u(\Lambda) &\coloneqq  \left\{ q \in M : \lim_{t\to-\infty} d_M(\phi^t(q), \Lambda) = 0 \right\},
\end{split}
\end{equation}
where $d_M$ denotes the Riemannian distance. The \textit{Stable Manifold Theorem} \cite{KatokHasselblatt95} ensures that the tangent bundles to $W^s(\Lambda)$ and $W^u(\Lambda)$, restricted to the orbit, satisfy the following identities:
\begin{equation}
    TW^s(\Lambda)|_\Lambda = \mathbb{R}V \oplus T^s_\Lambda \quad \text{and} \quad TW^u(\Lambda)|_\Lambda = \mathbb{R}V \oplus T^u_\Lambda.
\end{equation}

Since we intend to impose a constraint on the interaction between these invariant submanifolds, it is useful to recall the following definition.

\begin{definition}
\label{trasversity}
    Two submanifolds $L_1, L_2 \subset M$ intersect \textit{transversally} if at every $p \in L_1 \cap L_2$, the tangent spaces span the ambient space: $T_p M = T_p L_1 + T_p L_2$. Non-intersecting manifolds are transversal by convention.
\end{definition}
Note that we require the sum to be surjective, not necessarily direct. We can now state the formal definition.

\begin{definition}
\label{M-S-flow}
A flow $(\phi^t)_{t \in \mathbb{R}}$ generated by $V \in \mathfrak{X}(M)$ is called a \textbf{Morse--Smale flow} if:
\begin{enumerate}
    \item The non-wandering set $\mathbf{NW}(V)$, as per Definition \ref{nw}, consists of a finite union of hyperbolic critical elements (fixed points and closed orbits) $\left\{ \Lambda_i\right\}_{i=1}^K$.
    \item For any $i, j \in \{1, \dots, K\}$, the unstable manifold $W^u(\Lambda_j)$ and the stable manifold $W^s(\Lambda_i)$ intersect transversally. 
\end{enumerate}
\end{definition}

In essence, a Morse--Smale system has a simple recurrence structure (finite hyperbolic limit sets) and robust global geometry (transversal intersections). While the definition allows for both fixed points and closed orbits, specific examples may exhibit only one type of critical element, for instance, Morse--Smale gradient flows are Morse--Smale flows containing only fixed points.

\vspace{\baselineskip}

In the next section, we will make extensive use of the definition and properties of Morse--Smale flows. However, we will also need to introduce an additional technical hypothesis, which will be examined in the following definition.

\begin{definition}
\label{MS-cl_linearisable}
 Given a Morse--Smale flow $(\phi^t)_{t\in\mathbb{R}}$, it is said to be $C^l$-linearisable if, for each hyperbolic critical element $\Lambda_i$, one of the following conditions holds:
\begin{itemize}
    \item Denoting by $B_n(0, r) \subset \mathbb{R}^n$ an open ball centred at the origin, if $\Lambda_i$ is a fixed point, there exist a neighbourhood $W$ of $\Lambda_i$, a $C^l$-diffeomorphism $h: B_n(0, r) \to W$, and a linear map $A_i \in \mathrm{GL}(\mathbb{R}^n)$ such that the vector field $V$ satisfies
    \begin{equation}
        V \circ h = dh \circ L,
    \end{equation}
    where $L$ is the linearized vector field defined by $L(x) = (A_i x) \cdot \partial_x$, while $dh : TB_n(0, r) \to TW$ is the usual differential of a map between differentiable manifolds.

    \item If $\Lambda_i$ is a closed orbit with minimal period $T_{\Lambda_i}$, there exists a neighbourhood $W$ of $\Lambda_i$, a $C^l$-diffeomorphism $h: B_{n-1}(0, r) \times (\mathbb{R}/T_{\Lambda_i}\mathbb{Z}) \to W$, and a smooth map $\mathcal{A}_i: \mathbb{R}/T_{\Lambda_i}\mathbb{Z} \to M_{n-1}(\mathbb{R})$, where $M_{n-1}$ stands for the space of real square matrices of order $n-1$, such that the vector field $V$ satisfies
    \begin{equation}
        V \circ h = dh \circ L_{per},
    \end{equation}
    where the periodic model field $L_{per}$ is defined by $L_{per}(x, \theta) = (\mathcal{A}_i(\theta)x) \cdot \partial_x + \partial_\theta$, with $\theta$ being the coordinate on $\mathbb{R}/T_{\Lambda_i}\mathbb{Z}$.
\end{itemize}
A flow is said to be $C^\infty$-linearisable if it is $C^l$-linearisable for all $l \ge 1$.
\end{definition}

We conclude this overview by highlighting some structural features of this kind of flows. The first property shows how invariant manifolds organise the manifold $M$ into a precise decomposition, while the second one provides information on how the dimensions of the tangent spaces to the stable and unstable manifolds of two fixed points are related.

\begin{prop}
\label{prop: partition}
    Let $(\phi^t)_{t \in \mathbb{R}}$ be a Morse--Smale flow with critical elements $\{\Lambda_i\}_{i=1}^K$. The unstable and stable manifolds form disjoint partitions of $M$:
    \begin{equation}
        M = \bigsqcup_{i=1}^K W^u(\Lambda_i) = \bigsqcup_{j=1}^K W^s(\Lambda_j).
    \end{equation}
    Consequently, for every $p \in M$, there exists a unique pair of indices $(i, j)$ such that $p \in W^u(\Lambda_i) \cap W^s(\Lambda_j)$.
\end{prop}

\begin{proof}
    We refer to \cite[Lemma 3.4]{DangRiviere17Anisotropic} for the proof.
\end{proof}

\begin{prop}
\label{prop: growing-flow}
Let $\Lambda_i, \Lambda_j \in M$ be fixed points of a Morse--Smale flow as per Definition \ref{hyperbolic-point}. Then:
\begin{equation*}
W^{u}(\Lambda_{i}) \cap W^{s}(\Lambda_{j}) \neq \emptyset \implies \operatorname{dim}(W^{u}(\Lambda_{i})) > \operatorname{dim}(W^{u}(\Lambda_{j})).  
\end{equation*}
Moreover, the following dimensional identity holds:
\begin{equation}
\label{eq: dimension-intersection}
\begin{split}
\operatorname{dim}(T_{x}W^{u}(\Lambda_{i}))+\operatorname{dim}(T_{x}W^{s}(\Lambda_{j})) &= \operatorname{dim}(T_{x}W^{u}(\Lambda_{i}) +T_{x}W^{s}(\Lambda_{j}))\\
&\quad +\operatorname{dim}(T_{x}W^{u}(\Lambda_{i})\cap T_{x}W^{s}(\Lambda_{j})).
\end{split}
\end{equation}
\end{prop}

\begin{proof}
    For the proof of this statement, we refer to \cite[Lemma 3.5]{DangRiviere17Anisotropic}.
\end{proof}

\begin{rem}
    The first of the two statements in Proposition \ref{prop: growing-flow} contains a fundamental dynamical interpretation: the orbits starting in the proximity of fixed points with a stable manifold of a certain dimension must always end in the proximity of fixed points with stable orbits of a strictly higher dimension. This is a fundamental observation for the construction of the Morse complex, which we will see in the next section.
\end{rem}

\section{Torsion of a Complex and Thom--Smale Torsion}
\label{sec: Thom--Smale-torsion}

Our aim is to construct a global invariant associated with the dynamics of the flow and the topology of the bundle $\mathcal{E}$: the Thom--Smale torsion.

\subsection{Torsion of a Complex}
To provide a rigorous foundation, we first review the abstract definition of torsion for a generic finite-dimensional cochain complex, following \cite{Turaev01, Nicolaescu03, Mnev14}.

Consider a cochain complex $(C^\bullet, d)$ of finite-dimensional complex vector spaces:
\begin{equation*}
C^0 \xrightarrow{d_0} C^1 \xrightarrow{d_1} \dots \xrightarrow{d_{n-2}} C^{n-1} \xrightarrow{d_{n-1}} C^n.
\end{equation*}

The concept of torsion relies on the notion of the determinant line. For a vector space $V$ of dimension $n$, its determinant line is the top exterior power $\operatorname{Det}(V) \coloneqq \Lambda^n V$. This is a one-dimensional space generated by the wedge product of any basis of $V$.

\begin{definition}
\label{def: determinant-line}
The determinant line of the complex $(C^\bullet, d)$ is defined as the graded tensor product:
\begin{equation}
\operatorname{Det}(C^\bullet) \coloneqq \bigotimes_{k=0}^{n} \left(\operatorname{Det}(C^k)\right)^{(-1)^k},
\end{equation}
where $\operatorname{Det}(C^k)^{(-1)^k}$ denotes the space itself if $k$ is even, and its dual if $k$ is odd.
\end{definition}

Since the construction depends only on the graded vector space structure, we can specify this definition to the determinant line of the cohomology $H^\bullet(C^\bullet)$:
\begin{equation}
\operatorname{Det}(H^\bullet(C^\bullet)) = \bigotimes_{k=0}^{n} \left(\operatorname{Det}(H^k(C^\bullet))\right)^{(-1)^k}.
\end{equation}
A fundamental result in the theory, formalised, for instance, in \cite[Lemma 3.15]{Mnev14}, establishes the existence of a natural isomorphism between these two determinant lines. We call this map the algebraic torsion of the complex:
\begin{equation}
\label{eq:torsion_map}
\mathbb{T}: \operatorname{Det}(C^\bullet) \xrightarrow{\cong} \operatorname{Det}(H^\bullet(C^\bullet)).
\end{equation}
Since the determinant lines are one-dimensional vector spaces, this abstract isomorphism can be made concrete by choosing specific bases. If we fix non-zero elements $\mu_C \in \operatorname{Det}(C^\bullet)$ and $\mu_H \in \operatorname{Det}(H^\bullet(C^\bullet))$, the map $\mathbb{T}$ is uniquely identified by a nonvanishing scalar $\tau(C^\bullet) \in \mathbb{C}\backslash \{0\}$ such that 

\begin{equation}
\mathbb{T}(\mu_C) = \tau(C^\bullet) \cdot \mu_H. 
\end{equation}

The number $\tau(C^\bullet)$ is referred to as the numerical value of the torsion with respect to the chosen bases. In the special case where the complex is acyclic, \textit{i.e.}, when the cohomology vanishes, the target space is canonically $\mathbb{C}$, and $\tau(C^\bullet)$ depends solely on the basis of the complex $C^\bullet$.

\vspace{\baselineskip}

There exists an elegant way to evaluate the numerical part of the torsion concretely and, for that, it is useful to decompose the complex by mean of the following canonical splitting.

\begin{prop}
\label{generalised-hodge-dec}
Let $Z^k = \operatorname{Ker}(d_k)$ and $B^k = \operatorname{Im}(d_{k-1})$. Each space $C^k$ admits a decomposition:
\begin{equation}
C^k = B^k \oplus \mathscr{H}^k \oplus \mathscr{B}^{k+1},
\end{equation}
where $\mathscr{H}^k \cong H^k(C^\bullet)$ and $\mathscr{B}^{k+1}$ is a subspace mapped isomorphically onto $B^{k+1}$ by $d_k$.
\end{prop}

\begin{proof}
This decomposition follows from standard linear algebra arguments applied to the short exact sequences of the complex. For a detailed treatment, we refer to \cite[Section 3.1]{Turaev01}.    
\end{proof}

The previous decomposition allows us to define an operator that acts as a homotopy for the identity map, effectively inverting the differential on the acyclic part of the complex.

\begin{definition}
\label{contraction map}
Let $P^k: C^k \to \mathscr{H}^k$ be the projection onto the cohomology subspace defined by the splitting in Proposition \ref{generalised-hodge-dec}. A {chain contraction} is a collection of linear maps $\eta_k: C^{k+1} \to C^{k}$ satisfying:
\begin{equation}
    d_{k-1} \eta_{k-1} + \eta_k d_k = \mathbb{I}_{C^k} - P^k \quad \text{and} \quad \eta_{k-1} \circ \eta_k = 0.
\end{equation}
\end{definition}

The existence of such an operator is guaranteed for any finite-dimensional complex of vector spaces. We refer to \cite[Section 2.2]{Turaev01} for the explicit construction.

\begin{theorem}
Given a finite-dimensional complex $(C^\bullet, d)$ and the cohomology projector $P^k$, there exists a chain contraction $\{\eta\}_k$ as described in Definition \ref{contraction map}.
\end{theorem}

While the definition of $\eta$ is algebraic, a canonical candidate arises naturally upon introducing an inner product structure.

\begin{rem}
\label{rem:hodge-construction}
Equipping each $C^k$ with an inner product $\langle \cdot, \cdot \rangle_k$ allows us to define the adjoint $d^\dagger$ and the Laplacian $\Delta_k \doteq d_k^\dagger d_k + d_{k-1} d_{k-1}^\dagger$. The Hodge decomposition theorem then provides an orthogonal splitting:
\begin{equation}
\label{eq: generalised-Hodge-dec}
    C^k = \operatorname{Im}(d_{k-1}) \oplus_\bot \operatorname{Im}(d_k^\dagger) \oplus_\bot \operatorname{Ker}(\Delta_k).
\end{equation}
As in the infinite-dimensional setting, Hodge theory establishes an isomorphism between the cohomology $H^k(C^\bullet)$ and the space of harmonic forms $\operatorname{Ker}(\Delta_k)$ (see \cite[Proposition A.4]{Nicolaescu03}). This result is crucial as it provides a canonical choice for the subspace $\mathscr{H}^k$ in Proposition \ref{generalised-hodge-dec}: the projector $P^k$ is identified with the orthogonal projection onto the kernel of the Laplacian. 

Consequently, the Laplacian is invertible on the orthogonal complement $(\operatorname{Ker} \, \Delta_k)^\perp$. Its inverse is the {Green's operator} $G_k$, which satisfies the identity $\Delta_k G_k = \mathbb{I}_{C^k} - P^k$. A canonical chain contraction is then defined by:
\begin{equation}
\label{eq:eta-hodge}
\eta_k \coloneqq d_k^\dagger \, G_{k+1}.
\end{equation}
It is worth noting that a converse statement also holds: for every algebraic chain contraction $\eta$, there exists a specific inner product on the complex such that $\eta$ arises exactly as the canonical operator defined in Equation \eqref{eq:eta-hodge}.
\end{rem}

As already discussed, to compute the numerical value of the torsion, we must fix volume elements on both the determinant lines.

\begin{rem}
\label{rem: induced-product-cohomology}
Let us choose a basis $\{e_{k,i}\}$ for each $C^k$ which is orthonormal with respect to the inner product fixed in Remark \ref{rem:hodge-construction}\footnote{Fixing an inner product is equivalent to fixing a preferred basis up to unitary transformations. One could equivalently proceed in the reverse order: fix an arbitrary basis for the complex and declare it orthonormal, thus inducing an inner product that precisely determines the numerical value of the torsion.}. This  choice defines a volume element $\mu_C \in \operatorname{Det}(C^\bullet)$:
\begin{equation}
\label{eq: volume-element-C}
    \mu_C = \bigotimes_{k=0}^{n} (e_{k,1} \wedge \dots \wedge e_{k, \operatorname{dim} C^k})^{(-1)^k}.
\end{equation}
Furthermore, the inner product on $C^k$ induces a canonical metric structure on the cohomology group $H^k(C^\bullet)$. Using the isomorphism $H^k(C^\bullet) \cong \operatorname{Ker}(\Delta_k)$ described above, we define the inner product $\langle \cdot, \cdot \rangle_{H, k}$ on cohomology classes as:
\begin{equation}
    \langle [\alpha], [\beta] \rangle_{H, k} \coloneqq \langle \alpha_h, \beta_h \rangle_{C, k}, \quad \forall [\alpha], [\beta] \in H^k(C^\bullet),
\end{equation}
where $\alpha_h, \beta_h \in \operatorname{Ker}(\Delta_k)$ are the unique harmonic representatives of $[\alpha], [\beta] $. By choosing an orthonormal basis $\{h_{k,j}\}$ for the harmonic forms with respect to this induced metric, we define the volume element on the cohomology:
\begin{equation}
\label{eq: volume-element-H}
    \mu_H = \bigotimes_{k=0}^{n} (h_{k,1} \wedge \dots \wedge h_{k, \operatorname{dim}(H^k)})^{(-1)^k}.
\end{equation}
The algebraic torsion $\mathbb{T}$ maps $\mu_C$ to a multiple of $\mu_H$. This proportionality factor provides a concrete, albeit metric-dependent, numerical representation of the torsion.
\end{rem}

The abstract definition of torsion finds a concrete realisation through the spectrum of the Laplacian via the following theorem, whose proof can be found in \cite[Lemma 3.19]{Mnev14}.

\begin{theorem}
\label{thm:torsion-laplacian-formula}
Let $(C^\bullet, d)$ be the complex equipped with an inner product $\langle \cdot, \cdot \rangle_C$ and let $\Delta_k$ be the associated Laplacian. Consider the algebraic torsion map $\mathbb{T}$ as per Equation  \eqref{eq:torsion_map} and the canonical volume elements $\mu_C$ (Equation \eqref{eq: volume-element-C}) and $\mu_H$ (Equation \eqref{eq: volume-element-H}) defined by orthonormal bases. Then:
\begin{equation}
\mathbb{T}(\mu_C) = \tau(C^\bullet) \cdot \mu_H,
\end{equation}
where the scalar factor $\tau(C^\bullet)$ is given by:
\begin{equation}
\tau(C^\bullet) = \prod_{k=0}^{n} \left(\operatorname{det}'(\Delta_k)\right)^{(-1)^{k+1} k/2}.
\end{equation}
Here, $\operatorname{det}'(\Delta)$ denotes the product of all non-zero eigenvalues of the operator.
\end{theorem}

\begin{rem}
Note that since we are working with complex vector spaces, the choice of orthonormal bases for the cohomology defines the volume element $\mu_H$ only up to a phase factor $e^{i\theta}$. Consequently, the numerical torsion $\tau(C^\bullet)$ is inherently defined up to a complex number of modulus 1. This ambiguity is standard in the theory and will be implicitly understood in the following.
\end{rem}

\begin{rem}
\label{rem:invariance-of-torsion}
While the algebraic torsion $\mathbb{T}$ is canonical and entirely independent of any inner product, the scalar factor $\tau(C^\bullet)$ inherently depends on the chosen metric. This dependence arises because the auxiliary inner product determines the orthonormal volume elements $\mu_C$ and $\mu_H$ (see \cite[Section 8.4]{Mnev14}). Consequently, to extract a scalar from the isomorphism $\mathbb{T}$, one must fix specific volume elements for its one-dimensional domain $\operatorname{Det}(C^\bullet)$ and codomain $\operatorname{Det}(H^\bullet(C^\bullet))$. In practice, this is achieved by selecting preferred bases for both $C^k$ and $H^k(C^\bullet)$. By declaring these distinguished bases to be orthonormal, one unambiguously fixes the volume elements $\mu_C$ and $\mu_H$ required to evaluate the scalar factor $\tau(C^\bullet)$.
\end{rem}

Finally, we relate this spectral formula to the chain contraction operator $\eta$ constructed in Remark \ref{rem:hodge-construction}.
Recall that the complex decomposes into an acyclic part $C^k_{ac} \coloneqq \operatorname{Im}(d_{k-1}) \oplus \operatorname{Im}(d_k^\dagger)$ and a harmonic part. We restrict our attention to the acyclic subcomplex $(C^\bullet_{ac}, d)$ and define the operator $D_{ac} \coloneqq (d + \eta)|_{C^\bullet_{ac}}$.
Since $D_{ac}$ maps the even-graded subspace $C^{even}_{ac} = \bigoplus C^{2k}_{ac}$ to the odd-graded subspace $C^{odd}_{ac} = \bigoplus C^{2k+1}_{ac}$, and satisfies $(D_{ac})^2 = \mathbb{I}$, it constitutes a canonical isomorphism between the even and odd parts of the acyclic complex.
This allows us to define its determinant, denoted $\operatorname{det}(D_{ac}|_{C^{even}})$, by choosing bases for the two spaces.\footnote{It is important to distinguish this quantity from the spectral determinants appearing in Theorem \ref{thm:torsion-laplacian-formula}. Here, we are considering the linear map $D_{ac}$ between the spaces $C^{even}_{ac}$ and $C^{odd}_{ac}$. Since $D_{ac}$ is an isomorphism, these spaces have the same finite dimension; thus, the operator corresponds to a square matrix whose determinant is well-defined (up to the choice of bases).}

The relation between the spectral torsion and this operator is given by the following result, see \cite[Section 3.3]{Mnev14} for the proof.

\begin{prop}
\label{prop:torsion-contraction}
Let $\tau(C^\bullet)$ be the numerical torsion as per Theorem \ref{thm:torsion-laplacian-formula}. Its absolute value coincides with the determinant of the operator $d+\eta$ restricted to the even part of the acyclic subcomplex:
\begin{equation}
    |\tau(C^\bullet)| = |\operatorname{det}((d+\eta)|_{C^{even}_{ac}})|.
\end{equation}
\end{prop}

\subsection{Thom--Smale Torsion}
\label{subsec: TS-torsion}
Having recalled the general theory of torsion, we now proceed to construct a specific cochain complex associated with {a restricted class of Morse--Smale flows. Specifically, for the entirety of this subsection, we assume that the vector field, which we will denote by $X$ to distinguish it from the previous notation $V$, is of \textit{gradient type}. This implies that the non-wandering set of $X$ consists exclusively of fixed points, with no closed orbits present.}

We begin by constructing the standard Morse complex {with real coefficients, following the presentation of \cite{Hutchings02}. This standard structure will serve as the geometric foundation for the bundle-valued case, which we will subsequently introduce as a twisted version of the Morse complex.}

\begin{definition}
\label{def:morse_spaces}
 Let {$X$ be a gradient-type $C^\infty$-linearisable Morse--Smale vector field}. The set of critical points of index $k$ is denoted by {$\text{Crit}_k(X) \coloneqq \{p \in M \mid X_p = 0, \, \operatorname{dim} (W^s(p)) = k\}$}.
The $k$-th Morse space {$C^k_M(X)$} is the real vector space freely generated by these points:
\begin{equation}
    {C^k_M(X) \coloneqq \bigoplus_{p \in \text{Crit}_k(X)} \mathbb{R} \cdot p.}
\end{equation}
\end{definition}

To define a differential and turn the collection of Morse spaces into a complex, we must note the following facts. First, for every critical point $x \in \text{Crit}(X)$, we fix an arbitrary orientation of its stable manifold $W^s(x)$.
Consider now $p \in \text{Crit}_k(X)$ and $q \in \text{Crit}_{k+1}(X)$. The moduli space of trajectories flowing from $p$ to $q$ is defined as $\mathcal{M}(p,q) \coloneqq (W^u(p) \cap W^s(q))/\mathbb{R}$, which consists of a finite set of orbits due to the transversality condition.
For any trajectory $\gamma \in \mathcal{M}(p,q)$, the stable manifold of $p$ (dimension $k$) approaches the stable manifold of $q$ (dimension $k+1$). Therefore, the following isomorphism of oriented vector spaces holds:
\begin{equation}
    T W^s(p) \cong \mathbb{R}\dot{\gamma} \oplus T W^s(q),
\end{equation}
where $\mathbb{R}\dot{\gamma}$ carries the natural orientation given by the time direction. We define the sign of the trajectory, $\text{sgn}(\gamma) \in \{\pm 1\}$, to be $+1$ if the chosen orientation of $W^s(p)$ coincides with the product orientation of the flow direction and $W^s(q)$, and $-1$ otherwise.

We then define the counting number $n(p,q) \coloneqq \sum_{\gamma \in \mathcal{M}(p,q)} \text{sgn}(\gamma)$. The Morse differential $d_M^k: C^k_M(X) \to C^{k+1}_M(X)$ is the linear map given by:
\begin{equation}
\label{eq:morse_differential}
    d_M^k(p) \coloneqq \sum_{q \in \text{Crit}_{k+1}(V)} n(p,q) \cdot q.
\end{equation}

The fundamental property of this construction is stated below (see \cite[Lemma 2.2]{Hutchings02} for the proof).

\begin{prop}
\label{Morse-Complex}
The Morse differential satisfies the nilpotency condition $d_M^{k+1} \circ d_M^k = 0$. Consequently, $(C^\bullet_M(X), d_M)$ forms a cochain complex, called the Morse complex.
\end{prop}

The Morse complex constructed above relies solely on the dynamics of the vector field. To incorporate the data of the flat vector bundle $\mathcal{E}$ into this framework, we introduce a "twisted" version of the Morse complex. Following \cite{Chaubet22, Dang21}, we construct the Thom--Smale complex, where critical points are decorated with vectors from the fibres above them.

\begin{definition}
\label{Thom--Smale-complex}
The $k$-th \textit{Thom--Smale space} is defined as the direct sum of the fibres of $\mathcal{E}$ over the critical points of index $k$:
\begin{equation}
C^k_{TS}(X) \coloneqq \bigoplus_{p \in \text{Crit}_k(X)} \mathcal{E}_p.
\end{equation}
The \textit{Thom--Smale differential} $d_{TS}^k: C^k_{TS}(X) \to C^{k+1}_{TS}(X)$ is defined by summing parallel transport along flow lines. For a vector $u_p \in \mathcal{E}_p$:
\begin{equation}
\label{eq:TS_differential}
    d_{TS}^k(u_p) \coloneqq \sum_{q \in \text{Crit}_{k+1}(X)} \sum_{\gamma \in \mathcal{M}(p,q)} \text{sgn}(\gamma) \mathbf{P}_\gamma(u_p),
\end{equation}
where $\mathbf{P}_\gamma: \mathcal{E}_p \to \mathcal{E}_q$ is the parallel transport along the trajectory $\gamma$ with respect to the flat connection $\nabla$.
\end{definition}

The flatness of $\nabla$ ensures that parallel transport is homotopy invariant, making the sum well-defined. As shown in \cite[Section 10.2]{Chaubet22}, the operator satisfies $d_{TS}^{k+1} \circ d_{TS}^k = 0$, defining a valid cochain complex {$(C^\bullet_{TS}(X), d_{TS})$}.

{
\begin{rem}
\label{rem: only-gradient-flows}
If one were to consider a general Morse--Smale vector field containing closed orbits, the nilpotency condition $d_{TS}^2 = 0$ would generally fail. Geometrically, this failure is tied to the existence of heteroclinic orbits, namely, trajectories that belong to the intersection of the unstable manifold of a fixed point and the stable manifold of a closed orbit (or vice versa). Due to these connections, counting exclusively the flow lines between fixed points misses a portion of the dynamics, causing the identity $d_{TS}^2 = 0$ to break down. For a rigorous discussion of this phenomenon and to see how closed orbits are properly incorporated into this framework, we refer the interested reader to the work of S\'anchez-Morgado \cite{SM96}.
\end{rem}
}
In a chosen basis, the action of the Thom--Smale differential can be made explicit. Let $\{u_{p, j}\}_{j=1}^N$ be a basis for the fibre $\mathcal{E}_p$ at each critical point. Then:
\begin{equation}
\label{eq: Thom--Smale-differential}
    d_{TS}^k(u_{p, j}) = \sum_{q \in \text{Crit}_{k+1}(V)} \sum_{i=1}^N \left( \sum_{\gamma \in \mathcal{M}(p,q)} \text{sgn}(\gamma) [\mathbf{P}_\gamma]_{ij} \right) u_{q, i},
\end{equation}
where $[\mathbf{P}_\gamma]_{ij}$ are the matrix elements of the parallel transport. Since $(C^\bullet_{TS}(X), d_{TS})$ is a finite-dimensional differential complex, its numerical torsion $\tau(C^\bullet_{TS})$ is given accordingly to Theorem \ref{thm:torsion-laplacian-formula}.

{
\begin{rem}
While we have adopted a cochain perspective, which is naturally suited to parallel transport, another possible construction of the Thom--Smale complex is through chain operators. In that dual approach, the chain groups, typically denoted by $C_k(W^u, \mathcal{E}^*)$, are generated by the unstable manifolds of the critical points, and the boundary operator acts as $\partial: C_k(W^u, \mathcal{E}^*) \to C_{k-1}(W^u, \mathcal{E}^*)$. For the details of this formulation, we refer the reader to the works of Laudenbach \cite{Lau92} and Bismut--Zhang \cite{Bismut92}.
\end{rem}
}
\section{Analytic Torsion and Ruelle Zeta Function}
\label{sec: analytic-torsion-zeta}

This section introduces the two central objects involved in the celebrated Fried conjecture \cite{Fried87}: the Ray--Singer analytic torsion, a global topological invariant associated with the Riemannian metric, and the Ruelle Zeta function, a dynamical object constructed from the closed orbits of the flow. Despite their distinct origins, the conjecture posits a profound link between them, asserting that the analytic torsion should coincide with the value at zero of the Ruelle Zeta function. While originally proven for non-singular Morse--Smale flows in \cite{Fried87}, this result has recently been established for general Morse--Smale flows \cite{Shen21} and remains an active area of research for Anosov flows \cite{Dang20}.

{Before proceeding, we emphasise that from this section onwards, we drop the restrictive assumption of gradient-type flows adopted in Section \ref{subsec: TS-torsion}. In the general setting, an arbitrary Morse--Smale vector field $V$ may naturally exhibit both fixed points and closed orbits.}

\subsection{Analytic Torsion}
\label{sec: analytic-torsion}

The construction of the analytic torsion relies on the spectral theory of the twisted Laplacian $\Delta_k$ introduced in Definition \ref{def: codiff-laplacian}.

As established in \cite[Section 8]{Mnev14} and \cite{Seeley67}, this operator on a compact manifold is an elliptic self-adjoint operator with a purely discrete, non-negative spectrum:
\begin{equation}
    \text{spec}(\Delta_k) = \{0 \le \lambda_0 \le \lambda_1 \le \dots \to \infty \}.
\end{equation}
This spectral discreteness allows us to define the spectral zeta function:
\begin{equation}
\zeta(s, \Delta_k) \coloneqq \sum_{\lambda_j \in \text{spec}(\Delta_k) \setminus \{0\}} \lambda_j^{-s}.
\end{equation}
The series converges for $\mathrm{Re}(s) > n/2$ and admits a meromorphic continuation to the entire complex plane which is regular at $s=0$ (see \cite{Seeley67}). This regularity is crucial for the following definition.

\begin{definition}
\label{def: reg-det-laplacian}
The \textit{zeta-regularised determinant} of $\Delta_k$ (restricted to the orthogonal complement of its kernel) is defined as:
\begin{equation}
\operatorname{det}'(\Delta_k) \coloneqq \exp\left(-\frac{d\zeta (s, \Delta_k)}{ds}\Big|_{s=0}\right).
\end{equation}
\end{definition}

Since $M$ is compact, the twisted cohomology groups $H^k(M, \mathcal{E})$ are finite-dimensional vector spaces (see, e.g., \cite[Section 2.3]{DangRiviere17Topology}), allowing us to consider the determinant line of the cohomology:
\begin{equation}
    \text{Det}(H^\bullet(M, \mathcal{E})) \coloneqq \bigotimes_{k=0}^{n} \left(\text{Det}(H^k(M, \mathcal{E}))\right)^{(-1)^k}.
\end{equation}
To define the torsion as an element of this line, we need a reference volume element. The $L^2$-inner product on forms introduced in Definition \ref{scalar-product-forms} induces a metric on the harmonic forms $\mathcal{H}^k(M, \mathcal{E})$, and thus on the cohomology via the Hodge isomorphism $H^k(M, \mathcal{E}) \cong \mathcal{H}^k(M, \mathcal{E})$ presented in Theorem \ref{thm: hamonic-cohomology-iso}. By choosing an orthonormal basis $\{h_{k,j}\}$ for each $H^k(M, \mathcal{E})$, we construct the volume element:
\begin{equation}
    \eta_H \coloneqq \bigotimes_{k=0}^n \left( h_{k,1} \wedge \dots \wedge h_{k, b_k} \right)^{(-1)^k} \in \text{Det}(H^\bullet(M, \mathcal{E})).
\end{equation}

The Ray--Singer torsion is now defined as the combination of the regularised determinants of the Laplacian with the volume element of the cohomological determinant line.

\begin{definition}
\label{def: Ray--Singer-torsion}
The \textit{Ray--Singer analytic torsion} $\tau(M, \mathcal{E}) \in \text{Det}(H^\bullet(M, \mathcal{E}))$ is defined as:
\begin{equation}
    \tau(M, \mathcal{E}) \coloneqq \left( \prod_{k=0}^{n} \left(\operatorname{det}'(\Delta_k)\right)^{(-1)^{k+1} k/2} \right) \cdot \eta_H.
\end{equation}
\end{definition}

The topological nature of this object is established by the celebrated Cheeger-Müller theorem \cite{Cheeger79, Muller78}. This fundamental result asserts that the element $\tau(M, \mathcal{E}) \in \text{Det}(H^\bullet(M, \mathcal{E}))$ is an intrinsic invariant, completely independent of the auxiliary Riemannian metric $g$ on the manifold and the Hermitian metric $h$ on the bundle. While the scalar spectral factor and the canonical volume element $\eta_H$ both individually depend on the chosen metrics, their variations perfectly compensate each other. Consequently, the analytic torsion depends solely on the smooth structure of $M$ and the flat vector bundle $\mathcal{E}$.

\begin{rem}
\label{rem: scalar-component-torsion}
In the special case where the complex is acyclic, \textit{i.e.}, $H^k(M, \mathcal{E}) = 0$ for all $k \in \{0, \dots, n\}$, the determinant line is canonically $\mathbb{C}$, and the torsion reduces to a positive real number:
\begin{equation}
    \tau(M, \mathcal{E}) = \prod_{k=0}^{n} (\operatorname{det}' \Delta_k)^{(-1)^{k+1} k/2}.
\end{equation}
In the general non-acyclic case, the presence of the complex volume element $\eta_H$ introduces a phase ambiguity, which is intrinsic to the theory of complex vector bundles. Since the orthonormal basis defining $\eta_H$ is unique only up to a unitary transformation, the numerical torsion is defined up to a phase factor $e^{i\theta}$. 
\end{rem}

For calculations, it is often convenient to rewrite the torsion product in terms of the operators acting on coexact forms.

\begin{prop}
\label{prop:torsion-rearrangement-nonacyclic}
The analytic torsion can be equivalently expressed as:
\begin{equation}
\label{eq:rearranged-torsion-product}
    \tau(M, \mathcal{E}) = \left[ \prod_{k=1}^{n} \left(\operatorname{det}'(d^{\nabla, \dagger}_{k-1} d^\nabla_{k-1})\right)^{(-1)^{k+1}/2} \right] \cdot \eta_H,
\end{equation}
where $\operatorname{det}'(d^{\nabla, \dagger} d^\nabla)$ denotes the regularised determinant of the Laplacian restricted to the image of the adjoint operator.
\end{prop}

\begin{proof}
The proof relies on the Hodge decomposition of differential forms
\begin{equation*}
\Omega^{k}(M, \mathcal{E}) = \mathcal{H}^k(M, \mathcal{E}) \oplus \operatorname{Im}(d^\nabla_{k-1}) \oplus \operatorname{Im}(d_k^{\nabla, \dagger}),
\end{equation*}
as seen in Theorem \ref{thm:hodge-decomposition},
and on the spectral relation between the operators $d^\nabla d^{\nabla, \dagger}$ and $d^{\nabla, \dagger}d^\nabla$ restricted to the orthogonal complement of the harmonic forms. The Laplacian $\Delta$ respects this splitting and, since $\mathcal{H}^k(M, \mathcal{E})$ constitutes its kernel, the zeta-regularised determinant is defined by the restriction to the orthogonal complement. Specifically, on the subspace $\operatorname{Im}(d^\nabla_{k-1})$, the Laplacian acts as $d^\nabla_{k-1}d^{\nabla, \dagger}_{k-1}$, while on $\operatorname{Im}(d_k^{\nabla, \dagger})$, it acts as $d_k^{\nabla, \dagger}d_k^\nabla$. Consequently, the determinant factorises as:
\begin{equation}
\label{eq:det-split-detailed}
    \operatorname{det}' (\Delta_k) = \operatorname{det}' \left(d^\nabla_{k-1}d^{\nabla, \dagger}_{k-1}\Big|_{\operatorname{Im}(d^\nabla_{k-1})}\right) \cdot \operatorname{det}' \left(d_k^{\nabla, \dagger}d_k^\nabla\Big|_{\operatorname{Im}(d_k^{\nabla, \dagger})}\right).
\end{equation}
To relate the two terms appearing in the product, we consider the restriction of the exterior derivative $d_k^\nabla: \operatorname{Im}(d_k^{\nabla, \dagger}) \to \operatorname{Im}(d_k^\nabla)$. This map is an isomorphism: injectivity follows from the fact that $\operatorname{Ker}(d_k^\nabla) \cap \operatorname{Im}(d_k^{\nabla, \dagger}) = \{0\}$ due to the orthogonality of the Hodge decomposition, while surjectivity holds by definition. This isomorphism intertwines the operators, specifically $d_k^\nabla \circ (d_k^{\nabla, \dagger} d_k^\nabla) = (d_k^\nabla d_k^{\nabla, \dagger}) \circ d_k^\nabla$. This implies that the operator $d_k^{\nabla, \dagger} d_k^\nabla$ acting on $k$-forms and the operator $d_k^\nabla d_k^{\nabla, \dagger}$ acting on $(k+1)$-forms are isospectral on their respective domains. Thus, they share the same regularised determinant:
\begin{equation}
\label{eq:spectral-identity}
    \operatorname{det}' \left(d_k^\nabla d_k^{\nabla, \dagger}\Big|_{\operatorname{Im}(d_k^\nabla)}\right) = \operatorname{det}' \left(d_k^{\nabla, \dagger}d_k^\nabla\Big|_{\operatorname{Im}(d_k^{\nabla, \dagger})}\right).
\end{equation}
We can now substitute the factorisation presented in Equation \eqref{eq:det-split-detailed} into the definition of the torsion. By calling $D_k \coloneqq \operatorname{det}'(d_k^{\nabla, \dagger} d_k^\nabla)$ and using the identity written in Equation \eqref{eq:spectral-identity}, the term associated with $\operatorname{Im}(d^\nabla_{k-1})$ in the decomposition of $\Delta_k$ becomes $D_{k-1}$. The torsion product reads:
\begin{align*}
    \prod_{k=0}^{n} (\operatorname{det}' \Delta_k)^{p_k} &= \prod_{k=0}^{n} \left( D_{k-1} \cdot D_k \right)^{p_k} \\
    &= \left( \prod_{k=1}^{n} (D_{k-1})^{p_k} \right) \cdot \left( \prod_{k=0}^{n-1} (D_k)^{p_k} \right),
\end{align*}
where $p_k = (-1)^{k+1}k/2$. We re-index the first product by setting $j = k-1$. By gathering the terms over the common index range, the overall product reduces to
\begin{equation}
    \prod_{k=0}^{n} (\operatorname{det}' \Delta_k)^{p_k} = \prod_{k=0}^{n-1} (D_k)^{p_{k+1} + p_k},
\end{equation}
where the new exponent neatly simplifies as
\begin{align*}
    p_{k+1} + p_k &= \frac{(-1)^{k+2}(k+1)}{2} + \frac{(-1)^{k+1}k}{2} \\
    &= \frac{(-1)^{k}}{2} \left[ -(k+1) + k \right] = \frac{(-1)^{k+1}}{2}.
\end{align*}
Restoring the original definition of $D_k$ and shifting the index by one then immediately yields the stated formula:
\begin{equation}
    \prod_{k=0}^{n} (\operatorname{det}' \Delta_k)^{p_k} = \prod_{k=1}^{n} \left(\operatorname{det}'(d^{\nabla, \dagger}_{k} d^\nabla_{k})\right)^{(-1)^{k+1}/2}.
\end{equation}
\end{proof}

\begin{rem}
\label{rem: analytic-torsion-1}
This reformulation highlights that the scalar part of the torsion is determined by the spectrum on the acyclic part of the complex, while the cohomological contribution is isolated in $\eta_H$. A notable consequence is that if the manifold $M$ has even dimension $n$, the product is trivial. Indeed, by Poincaré duality, the contributions cancel out, and the numerical part of the torsion is identically $1$.
\end{rem}

\subsection{Ruelle Zeta Function}
\label{sec: Ruelle-zeta}

The Ruelle Zeta function is a dynamical invariant constructed as a formal product over the closed orbits of the flow, weighting each trajectory by its period and the holonomy of the given flat bundle. In the Morse--Smale setting, this construction is particularly straightforward: since the non-wandering set contains only finitely many closed orbits, the zeta function reduces to a finite product of determinants. To formally define these contributions, we must first interpret the unitary representation $\rho$ and the orientation of the unstable manifolds in terms of geometric transport along these orbits.

\begin{rem}
\label{rem:trivial-connection-recall}
Recall that the flat vector bundle $\mathcal{E}$ is constructed as a quotient of the trivial bundle $\widetilde{M} \times E$ over the universal cover of $M$. This trivial bundle carries a canonical flat connection $\mathbf{d}$ which acts on a section $\widetilde{s} \equiv f: \widetilde{M} \to E$ simply as the standard exterior derivative, $\mathbf{d}\widetilde{s} \coloneqq df$. Consequently, parallel sections are exactly the constant functions, meaning that parallel transport along any path in $\widetilde{M}$ acts as the identity on the vector component $E$.
\end{rem}

\begin{prop}
\label{flat-canonical-connection}
The canonical flat connection on $\widetilde{M} \times E$ descends to a flat connection $\nabla$ on $\mathcal{E} \to M$. A section $s \in \Gamma(M, \mathcal{E})$ is parallel along a path $\gamma$ if and only if its lift $\widetilde{s}$ to the universal cover is parallel along any lift $\widetilde{\gamma}$ with respect to the trivial connection $\mathbf{d}$.
\end{prop}

This construction allows us to identify the holonomy of the connection with the representation of the fundamental group.

\begin{definition}
\label{parallel-transport-loop}
    Let $\gamma: [0,1] \to M$ be a loop based at $p$. The \textit{holonomy map} $\mathbf{P}_\gamma: \mathcal{E}_p \to \mathcal{E}_p$ is the linear isomorphism induced by parallel transport along $\gamma$ via the connection $\nabla$.
\end{definition}

The following theorem formalises the link between this geometric transport and the algebraic representation.

\begin{theorem}
\label{th:representation-parallel}
   Let $\gamma$ be a loop based at $p$ representing the class $[\gamma] \in \pi_1(M, p)$. The action of the holonomy map $\mathbf{P}_\gamma$ on the fibre $\mathcal{E}_p$ coincides with the action of the representation $\rho([\gamma]) \in \mathrm{GL}(E)$.
\end{theorem}

\begin{proof}
   Since the connection is flat, the parallel transport depends only on the homotopy class of the loop. Let $\widetilde{\gamma}: [0,1] \to \widetilde{M}$ be the unique lift\footnote{Given a path $\gamma: [0,1] \to M$ and a starting point $\widetilde{p} \in \pi^{-1}(\gamma(0))$, a \textit{lift} is the unique continuous path $\widetilde{\gamma}: [0,1] \to \widetilde{M}$ satisfying $\widetilde{\gamma}(0) = \widetilde{p}$ and $\pi \circ \widetilde{\gamma} = \gamma$. The existence and uniqueness are guaranteed by the lifting properties of covering spaces, see \cite[Lemma 54.1]{Munkres00}.} of $\gamma$ starting at $\widetilde{p} \in \pi^{-1}(p)$. Consequently, the endpoint satisfies $\widetilde{\gamma}(1) = \widetilde{p} \cdot [\gamma]$.
   
   Knowing that an element $u \in \mathcal{E}_p$ is represented by the equivalence class $[(\widetilde{p}, v)]$, Proposition \ref{flat-canonical-connection} states that the parallel transport of $u$ along $\gamma$ corresponds to parallel transport of the lift $(\widetilde{p}, v)$ along $\widetilde{\gamma}$ using the trivial connection $\mathbf{d}$. Since $\mathbf{d}$ keeps the vector component constant (Remark \ref{rem:trivial-connection-recall}), the transported element at the endpoint is $(\widetilde{\gamma}(1), v) = (\widetilde{p} \cdot [\gamma], v)$, which, according to the definition of the quotient space $\mathcal{E}$ (Definition \ref{total-space}), is equivalent to:
\begin{equation}
   (\widetilde{p} \cdot [\gamma], v) \sim (\widetilde{p}, \rho([\gamma]) v).
\end{equation}
   Thus, the result of the transport corresponds to the class $[(\widetilde{p}, \rho([\gamma]) v)]$, which is precisely the action of $\rho([\gamma])$ on the original vector $v$ in the fibre.
\end{proof}

To define the Ruelle Zeta function, we require one final geometric ingredient related to the orientation of the unstable manifolds along closed orbits.

\begin{definition}
\label{twist-factor}
Let $V$ be a $C^\infty$-linearisable Morse--Smale vector field. For a closed orbit $\Lambda$, the \textit{twist factor} $\Delta(\Lambda) \in \{+1, -1\}$ is defined as $+1$ if the unstable manifold $W^u(\Lambda)$ is orientable, and $-1$ otherwise.
\end{definition}

We now have all the elements to define the function directly in terms of the spectral data of the closed orbits, following the conventions of \cite{Shen21}.

\begin{definition}
\label{Ruelle-zeta}
The {Ruelle Zeta function} of a $C^\infty$-linearisable Morse--Smale vector field is defined as the product over the closed orbits $\Lambda$:
\begin{equation}
\label{eq:Ruelle-zeta}
    \mathfrak{R}_{V, \rho}(z) \coloneqq \prod_{\Lambda \, \text{closed}} \operatorname{det}\left(\mathbb{I} - \Delta(\Lambda) \rho(\Lambda)e^{-T_{\Lambda}z}\right)^{(-1)^{\mathrm{ind}(\Lambda)}},
\end{equation}
where $T_{\Lambda}$ is the minimal period, $\Delta(\Lambda)$ is the twist factor defined in Definition \ref{twist-factor}, and $\mathrm{ind}(\Lambda)$ is the rank of the unstable bundle $T^u_\Lambda$.
\end{definition}

The Ruelle zeta function enjoys the following analytical property (c.f. \cite[Section 4]{Shen21Survey}).

\begin{prop}
\label{prop:ruelle-zeta-well-defined-ms}
There exists a real parameter $\sigma_0 > 0$ such that the product defining $\mathfrak{R}_{V,\rho}(z)$ converges to a non-zero analytic function for $\mathrm{Re}(z) > \sigma_0$. Furthermore, this function admits a meromorphic continuation to the entire complex plane that is analytic at $z=0$.
\end{prop}

Since our ultimate goal involves the value of this function at the origin, we will implicitly refer to its meromorphic extension. It is crucial to observe that we must assume a generic spectral condition on the dynamics: every closed orbit $\Lambda$ must be \textit{non-aligned}, meaning that the twist factor $\Delta(\Lambda)$ is not an eigenvalue of the holonomy representation $\rho(\Lambda)$. This hypothesis guarantees that the linear maps $\left(\mathbb{I} - \Delta(\Lambda)\rho(\Lambda)\right)$ are all invertible, preventing the factors in Equation \eqref{eq:Ruelle-zeta} from vanishing at $z=0$ and thus ensuring that $\mathfrak{R}_{V,\rho}(0)$ is a well-defined, non-vanishing, complex number.

\section{Fried Conjecture}
\label{sec: Fried-conjecture}

The Fried conjecture, originally posited for Axiom-A flows, becomes a proven theorem in the context of Morse--Smale flows. While Fried initially proved it for non-singular flows \cite{Fried87}, the result holds in full generality. However, this extension requires upgrading the statement from a simple scalar equality to an equality between two metrics constructed on the determinant line of the twisted cohomology complex. The exposition in this section relies primarily on the comprehensive framework developed to this end in \cite{Shen21}.

The first step is to construct a canonical metric on the determinant line of the cohomology for a flat vector bundle over the circle $\mathbb{S}^1$. This foundational result serves as the essential local model for understanding the contribution of hyperbolic closed orbits.

\begin{definition}
\label{def: A}
Given the flat vector bundle $(\mathbb{S}^1, \mathcal{E}, \mathbb{P})$ with representation $\rho$, and fixing a generator $\alpha_0 \in \pi_1(\mathbb{S}^1)$ compatible with the standard orientation, we define the holonomy matrix:
\begin{equation}
A \coloneqq \rho(\alpha_0) \in \mathrm{GL}(E).
\end{equation}
\end{definition}

\begin{rem}
\label{rem: complex-from-triangulation}
As detailed in \cite[pages 5-6]{Shen21}, the cohomology of $\mathbb{S}^1$ with coefficients in $\mathcal{E}$ is computed via a minimal triangulation. The resulting cochain complex $(C^\bullet(\mathbb{S}^1, \mathcal{E}), d)$ is given by:
\begin{equation*}
0 \to C^0(\mathbb{S}^1, \mathcal{E}) \xrightarrow{d^0 = A - \mathbb{I}} C^1(\mathbb{S}^1, \mathcal{E}) \to 0,
\end{equation*}
where the spaces are canonically identified with the fibre $E$. The cohomology groups are $H^0(C^\bullet(\mathbb{S}^1, \mathcal{E})) \cong \operatorname{Ker}(A - \mathbb{I})$ and $H^1(C^\bullet(\mathbb{S}^1, \mathcal{E})) \cong E / \operatorname{Im}(A - \mathbb{I})$ and the determinant line is defined as:
\begin{equation}
\text{Det}(H^\bullet(C^{\bullet}(\mathbb{S}^1, \mathcal{E}))) \coloneqq \text{Det}(H^0(C^\bullet(\mathbb{S}^1, \mathcal{E}))) \otimes \left(\text{Det}(H^1(C^\bullet(\mathbb{S}^1, \mathcal{E})))\right)^{-1}.
\end{equation}
\end{rem}

Let us now equip this determinant line with a specific norm.

\begin{definition}
\label{def: norm-det-line-circle}
Using the algebraic torsion isomorphism $\mathbb{T}$ from Equation \eqref{eq:torsion_map} and a volume element $\mu_{A, C} \in \text{Det}(C^{\bullet}(\mathbb{S}^1, \mathcal{E}))$ induced by an inner product, we equip the determinant line with a norm $\|\cdot \|_{\text{Det}(H^\bullet(C^{\bullet}(\mathbb{S}^1, \mathcal{E})))}$ subject to the normalisation condition:
\begin{equation}
    \|\mathbb{T}(\mu_{A, C}) \|_{\text{Det}(H^\bullet(C^{\bullet}(\mathbb{S}^1, \mathcal{E})))} = 1.
\end{equation}
\end{definition}

In the acyclic case, this norm admits an explicit algebraic expression linking it to the dynamical Zeta function.

\begin{prop}
\label{prop: norm-of-one}
If $1 \notin \text{spec}(A)$, the complex $(C^\bullet(\mathbb{S}^1, \mathcal{E}), d)$ is acyclic. Hence, the determinant line is canonically isomorphic to $\mathbb{C}$, and the norm of the generator $\mu_{A, H} \in \text{Det}(H^\bullet(C^{\bullet}(\mathbb{S}^1, \mathcal{E})))$ is given by \cite[Section 1.2]{Shen21}:
\begin{equation}
\|\mu_{A, H} \|_{\text{Det}(H^\bullet(C^{\bullet}(\mathbb{S}^1, \mathcal{E})))} = |\operatorname{det}(\mathbb{I} - A)|^{-1}.
\end{equation}
\end{prop}

Note that reversing the orientation of $\mathbb{S}^1$ replaces $A$ with $A^{-1}$, modifying the norm to $|\operatorname{det}(\mathbb{I} - A^{-1})|^{-1}$.

Before transposing this construction to the case of closed orbits of a Morse--Smale flow, we recall the notion of relative cohomology, which will be useful to define the Milnor norm.

\begin{definition}
\label{def: relative-cohomology}
For a smooth map between manifolds $f: S \to M$, we define $\Omega^\bullet(f, \mathcal{E})$ as:
\begin{equation}
    \Omega^k(f, \mathcal{E}) \coloneqq \Omega^k(M, \mathcal{E}) \oplus \Omega^{k-1}(S, \mathcal{E}),
\end{equation}
equipped with the differential $d(\omega, \theta) \coloneqq (d^M_k \omega, f^*\omega - d^S_{k-1}\theta)$.
This fits into the short exact sequence of complexes:
\begin{equation*}
    0 \to \Omega^{\bullet-1}(S, \mathcal{E}) \xrightarrow{\; \alpha \;} \Omega^\bullet(f, \mathcal{E}) \xrightarrow{\; \beta \;} \Omega^\bullet(M, \mathcal{E}) \to 0,
\end{equation*}
where the chain maps are defined by $\alpha(\theta) = (0, \theta)$ and $\beta(\omega, \theta) = \omega$. This sequence induces the following long exact sequence in cohomology:
\begin{equation*}
\dots \xrightarrow{f^*} H^{k-1}(S, \mathcal{E}) \xrightarrow{\;\alpha^* \;} H^k(f, \mathcal{E}) \xrightarrow{\;\beta^* \;} H^k(M, \mathcal{E}) \xrightarrow{\;f^* \;} H^{k}(S, \mathcal{E}) \to \dots
\end{equation*}
where $\alpha^*$ and $\beta^*$ are the induced maps in cohomology, and the connecting homomorphism coincides with the pullback map $f^*$.
If $\iota: S \hookrightarrow M$ is an immersion, the relative cohomology $H^k(M, S, \mathcal{E})$ is defined as the cohomology $H^k(\iota, \mathcal{E})$ of the corresponding mapping cone.
\end{definition}

To construct the Milnor metric, we analyse the algebraic structure of the cohomology by decomposing the determinant line via a Smale filtration. This approach exploits the defining property of Morse--Smale flows: the non-wandering set consists of a finite number of hyperbolic critical elements.

\begin{definition}
\label{Smale-filtration}
Let $V$ be a $C^\infty$-linearisable Morse--Smale vector field with non-wandering set $\mathbf{NW}(V) = \{\Lambda_j\}_{j=1}^K$. A \textit{Smale filtration} is an increasing sequence of compact submanifolds with boundary:
\begin{equation*}
\emptyset = M^0 \subset M^1 \subset \dots \subset M^K = M,
\end{equation*}
such that for each $j$, the difference $M^{j} \setminus M^{j-1}$ contains exactly one critical element $\Lambda_j$, satisfying $\{\Lambda_j\} = \bigcap_{t \in \mathbb{R}} \phi^t(M^{j} \setminus M^{j-1})$.
\end{definition}

This filtration allows us to decompose the determinant line of the total cohomology into a tensor product of contributions associated with each critical element, see \cite[Section 2.4, Equation (2.24)]{Shen21}.

\begin{prop}
\label{prop: canonical-iso-sigma}
Let $\{M^j\}_{j=0}^K$ be a Smale filtration. There exists a canonical isomorphism of determinant lines:
\begin{equation}
    \sigma_V: \bigotimes_{j=1}^{K} \text{Det}(H^\bullet(M^{j}, M^{j-1}, \mathcal{E})) \stackrel{\cong}{\longrightarrow} \text{Det}(H^\bullet(M, \mathcal{E})),
\end{equation}
where $\text{Det}(H^\bullet(M^{j}, M^{j-1}, \mathcal{E}))$ denotes the determinant line of the relative cohomology of the pair $(M^j, M^{j-1})$.
\end{prop}

This isomorphism reduces the problem of defining a metric on the global determinant line to defining metrics on the relative cohomology factors, the structure of which is entirely determined by the nature of the critical element $\Lambda_j$.

\begin{rem}
\label{rem: relative-cohomology-filtration}
Following \cite{Franks82}, the structure of the relative cohomology $H^\bullet(M^{j}, M^{j-1}, \mathcal{E})$ depends entirely on the nature of the critical element $\Lambda_j$. If $\Lambda_j$ is a hyperbolic fixed point with index $\mathrm{ind}(\Lambda_j) = \operatorname{dim}(W^u(\Lambda_j))$, the relative cohomology is concentrated in a single degree:
\begin{equation}
\label{eq: relative-cohomology-point}
    H^q(M^j, M^{j-1}, \mathcal{E}) \cong 
    \begin{cases}
        \mathcal{E}_{\Lambda_j} & \text{if } q = \mathrm{ind}(\Lambda_j), \\
        0 & \text{otherwise.}
    \end{cases}
\end{equation}
Consequently, the following canonical isomorphism holds:
\begin{equation}
    \text{Det}(H^\bullet(M^j, M^{j-1}, \mathcal{E})) \cong (\text{Det}(\mathcal{E}_{\Lambda_j}))^{(-1)^{\mathrm{ind}(\Lambda_j)}}.
\end{equation}
Conversely, if $\Lambda_j$ is a hyperbolic closed orbit with index $\mathrm{ind}(\Lambda_j) = \text{rank}(T^u_{\Lambda_j})$, the relative cohomology identifies with the cohomology of the circle $\Lambda_j$ with coefficients in a specific bundle:
\begin{equation}
\label{eq: relative-cohomology-orbit}
    H^q(M^{j}, M^{j-1}, \mathcal{E}) \cong 
    \begin{cases}
        H^{q-\mathrm{ind}(\Lambda_j)}(\Lambda_j, o(T^u_{\Lambda_j}) \otimes \mathcal{E}|_{\Lambda_j}) & \text{if } q - \mathrm{ind}(\Lambda_j) \in \{0, 1\}, \\
        0 & \text{otherwise.}
    \end{cases}
\end{equation}
Here, the twisting bundle is the tensor product of the flat bundle restricted to the orbit, $\mathcal{E}|_{\Lambda_j}$, and the orientation line bundle of the unstable bundle, $o(T^u_{\Lambda_j})$. Since $\Lambda_j$ is diffeomorphic to $\mathbb{S}^1$, we can apply the results established for the circle model in Proposition \ref{prop: norm-of-one}. The holonomy of the bundle $o(T^u_{\Lambda_j}) \otimes \mathcal{E}|_{\Lambda_j}$ corresponds to the product of the bundle holonomy $\rho(\Lambda_j)$ and the orientation twist $\Delta(\Lambda_j)$ defined in Equation \eqref{twist-factor}. Thus, the relevant holonomy matrix is $A_{\Lambda_j} = \Delta(\Lambda_j)\rho(\Lambda_j)$. Crucially, the relative cohomology groups vanish if and only if $1 \notin \text{spec}(\Delta(\Lambda_j)\rho(\Lambda_j))$. This property coincides exactly with the non-aligned condition assumed for the Ruelle Zeta function, ensuring that the local contribution from closed orbits is acyclic.
\end{rem}

Eventually, the isomorphism $\sigma_V$ is used to transport the local metrics defined on the graded pieces of the filtration to the global determinant line.

\begin{definition}[\textbf{The Milnor Metric}]
\label{def: milnor-metric}
Consider the canonical isomorphism $\sigma_V$ established in Proposition \ref{prop: canonical-iso-sigma}. We define the {Milnor metric}
\begin{equation}
\| \cdot \|_{M,V}: \text{Det}(H^\bullet(M, \mathcal{E})) \to \mathbb{R},
\end{equation}
as the unique norm that renders $\sigma_V$ an isometry. This is tantamount to equipping each factor in the tensor product domain of $\sigma_V$ with a specific metric:
\begin{itemize}
    \item For factors associated with fixed points, the metric induced by the Hermitian product on the fibres of $\mathcal{E}$;
    \item For factors associated with closed orbits, the algebraic norm as defined in Definition \ref{def: norm-det-line-circle} $\|\cdot\|_{\text{Det} \, H^\bullet(\Lambda_j, o(T^u_{\Lambda_j})\otimes \mathcal{E}|_{\Lambda_j})}$.
\end{itemize}
\end{definition}

\begin{rem}
\label{rem:milnor-well-defined}
One might ask about the existence, uniqueness, and well-definedness of such a metric. 
Algebraically, existence and uniqueness are guaranteed since the domain of $\sigma_V$ is a tensor product of lines equipped with specific metrics, therefore it carries a unique induced product metric. As $\sigma_V$ is a linear isomorphism between one-dimensional vector spaces, there exists exactly one norm on the codomain that renders the map an isometry.
A more subtle issue concerns the independence of this definition from the construction of $\sigma_V$, specifically regarding the order in which the relative cohomology groups are "fused" along the Smale filtration. As established in \cite[Section 1.3]{Shen21}, the fusion of determinant lines is associative and commutative up to a sign. Since the metric depends only on the modulus, it is insensitive to these sign ambiguities, rendering the Milnor metric well-defined. Furthermore, it can be shown that the resulting metric does not depend on the specific choice of the Smale filtration used (see \cite[Remark 2.15]{Shen21}).
\end{rem}

\begin{rem}
\label{rem: milnor_construction}
The factorisation between critical elements contributions is intrinsic to this definition. Consider an element $\lambda \in \text{Det}(H^\bullet(M, \mathcal{E}))$. Its preimage decomposes uniquely as $\sigma_V^{-1}(\lambda) = v_1 \otimes \dots \otimes v_{K}$, where each $v_j$ belongs to the local determinant line associated with the critical element $\Lambda_j$. The Milnor norm is defined as:
\begin{equation}
    \|\lambda\|_{M,V} \coloneqq \prod_{j=1}^{K} \|v_j\|_{\Lambda_j}.
\end{equation}
The local norms depend on the type of critical element:
\begin{itemize}
    \item If $\Lambda_j$ is a \textit{fixed point}, the line is $(\operatorname{det} \mathcal{E}_{\Lambda_j})^{(-1)^{\mathrm{ind}(\Lambda_j)}}$. Here, $\|v_j\|_{\Lambda_j}$ is the standard norm induced by the Hermitian metric $h$ on the fibre.
    \item If $\Lambda_j$ is a \textit{closed orbit}, the line is canonically trivial ($\cong \mathbb{C}$ or its dual). Here, $\|v_j\|_{\Lambda_j}$ is computed using the metric defined in Definition \ref{def: norm-det-line-circle}, normalised by the determinant of the Poincaré map $(\mathbb{I} - A_{\Lambda_j})$.
\end{itemize}
\end{rem}

This separation leads to a fundamental structural decomposition of the determinant line into a static part of the flow and a recurrent part. We can regroup the terms in the tensor product into two disjoint sets: $\lambda_A(V, \mathcal{E})$ containing the contributions from fixed points, and $\lambda_B(V, \mathcal{E})$ containing the contributions from closed orbits.
\begin{equation}
   \text{Det}( H^\bullet(M, \mathcal{E})) \cong \lambda_A(V, \mathcal{E}) \otimes \lambda_B(V, \mathcal{E}).
\end{equation}
The fixed-point line $\lambda_A(V, \mathcal{E})$ can be further organised by the Morse index of the critical points:
\begin{equation}
\label{eq: lambda-det-line}
\lambda_A(V, \mathcal{E}) \coloneqq \bigotimes_{k=0}^{n} \left( \bigotimes_{p \in \text{Crit}_{n-k}(V)} \text{Det}(\mathcal{E}_{p}) \right)^{(-1)^{n-k}},
\end{equation}
{which, as a one-dimensional vector space, coincides with the determinant line of the Thom--Smale cochains introduced in Definition \ref{Thom--Smale-complex}. However, it is crucial to remember that for a general Morse--Smale flow containing closed orbits, these spaces do not form a valid cochain complex, so one cannot define an algebraic torsion on them.}

{Regarding the closed orbits, the determinant line of the relative cohomology associated with each orbit $\Lambda_j$ is canonically isomorphic to $\mathbb{C}$. Let $1_{\Lambda_j} \in \operatorname{Det}(H^\bullet(M^j, M^{j-1}, \mathcal{E}))$ denote the canonical generator corresponding to $1 \in \mathbb{C}$. Consequently, their tensor product $\lambda_B(V, \mathcal{E})$ is also canonically isomorphic to $\mathbb{C}$ and is generated by $\mathbf{1}_B \coloneqq \bigotimes_{\Lambda_j} 1_{\Lambda_j}$. Applying the norm from Definition \ref{def: norm-det-line-circle} to each $1_{\Lambda_j}$, the norm of $\mathbf{1}_B$ evaluates to:
\begin{equation}
\label{eq: Milnor-contribution-orbit}
    \|\mathbf{1}_B\| = \prod_{\Lambda_j \in \text{Orbits}} \|1_{\Lambda_j}\|_{\Lambda_j} = \prod_{\Lambda_j} |\operatorname{det}(\mathbb{I} - \Delta(\Lambda_j)\rho(\Lambda_j))|^{(-1)^{\mathrm{ind}(\Lambda_j)+1}}.
\end{equation}
Comparing this expression with Definition \ref{Ruelle-zeta}, this product corresponds exactly to the inverse of the absolute value at zero of the Ruelle Zeta function:
\begin{equation}
    \|\mathbf{1}_B\| = |\mathfrak{R}_{V, \rho}(0)|^{-1}.
\end{equation}}
See \cite[Proposition 2.16]{Shen21} for the detailed derivation.

\begin{rem}
\label{rem: milnor-components-interpretation}
{We can now synthesise these results to properly understand the factorisation of the Milnor metric. Given a unitary volume element $\mu_H \in \text{Det}(H^\bullet(M, \mathcal{E}))$, its preimage through $\sigma_V$ uniquely decomposes as $\sigma_V^{-1}(\mu_H) = \mu_A \otimes \mathbf{1}_B$, where $\mu_A \in \lambda_A(V, \mathcal{E})$ and $\mathbf{1}_B \in \lambda_B(V, \mathcal{E})$ is the canonical generator, as previously presented. Consequently, the Milnor norm of $\mu_H$ is given by the product of the metrics on the two determinant lines:
\begin{equation}
\label{eq: factorisation-Milnor}
    \|\mu_H\|_{M,V} = \|\mu_A\|_{\lambda_A} \cdot |\mathfrak{R}_{V, \rho}(0)|^{-1}.
\end{equation}
Although the line $\lambda_A(V, \mathcal{E})$ is isomorphic to the determinant line of the Thom--Smale cochains, the scalar $\|\mu_A\|_{\lambda_A}$ does not correspond to the algebraic torsion of the Thom--Smale complex, which is undefined in this general setting, as discussed in Remark \ref{rem: only-gradient-flows}. Therefore, the element $\mu_A$ intrinsically encodes all the dynamical interactions between fixed points and closed orbits, \textit{i.e.}, the heteroclinic intersections.\\
The expression presented in Equation \eqref{eq: factorisation-Milnor} reduces to two well-known decoupled regimes when the dynamics are simpler:
\begin{itemize}
    \item \textbf{Purely gradient flows:} If the vector field $X$ has no closed orbits, the sector $\lambda_B$ is empty. The isomorphism $\sigma_V$ coincides with the algebraic torsion isomorphism of the Thom--Smale complex, yielding $\|\mu_H\|_{M,X} = |\tau(C^\bullet_{TS}(X))|$. As noted in \cite[Remark 2.13]{Shen21}, this construction exactly recovers the metric defined by Bismut and Zhang \cite[Definition 1.9, Theorem 7.1]{Bismut92}.
    \item \textbf{Fixed-point-free flows:} If the flow has no fixed points, the sector $\lambda_A$ is empty. The Milnor norm depends entirely on the closed orbits, yielding $\|\mu_H\|_{M,V} = |\mathfrak{R}_{V, \rho}(0)|^{-1}$, thus recovering the original setting of Fried's conjecture, as proven in \cite[Proposition 0.1]{Shen21}.
\end{itemize}}
\end{rem}

Having constructed the Milnor metric from the dynamical data, we now turn to its analytic counterpart. The Ray--Singer metric combines the spectral determinant of the Laplacian with the $L^2$-structure on the cohomology.

\begin{definition}[\textbf{The Ray--Singer Metric}]
\label{def: Ray--Singer-metric}
The \textit{Ray--Singer metric} $\|\cdot\|_{RS}$ on the determinant line $\text{Det}(H^\bullet(M, \mathcal{E}))$ is defined as:
\begin{equation}
    \| \cdot \|_{RS} \coloneqq \tau(M, \mathcal{E}) \cdot \| \cdot \|_{\text{Det}} \, ,
\end{equation}
where $\| \cdot \|_{\text{Det}}: \text{Det}(H^\bullet(M, \mathcal{E})) \to \mathbb{R}$ is the norm induced by the Riemannian and Hermitian metrics, and $\tau(M, \mathcal{E})$ is the real scalar defined by the regularised determinants of the Laplacians as per Remark \ref{rem: scalar-component-torsion}.
\end{definition}

The following result represents the generalisation of Fried's conjecture for the class of Morse--Smale flows. We refer to \cite[Theorem 3.12]{Shen21} for the complete proof.

\begin{theorem}[\textbf{Fried Conjecture}]
\label{thm: Fried-conjecture}
Let $V$ be a $C^\infty$-linearisable Morse--Smale vector field, as per Definition \ref{MS-cl_linearisable}. Let $\|\cdot\|_{M,V}$ be the Milnor metric defined in Definition \ref{def: milnor-metric} and let $\|\cdot\|_{RS}$ be the Ray--Singer metric as per Definition \ref{def: Ray--Singer-metric}. Then, the two metrics coincide:
\begin{equation}
    \| \lambda \|_{M,V} = \| \lambda \|_{RS} \, , \quad \forall \lambda \in \text{Det}(H^\bullet(M, \mathcal{E})).
\end{equation}
\end{theorem}

A necessary clarification regarding the validity of this result is in order. {As mentioned at the beginning of this survey, the strict equality stated in Theorem \ref{thm: Fried-conjecture} relies on the representation $\rho$ being unitary. This assumption indeed allows for the choice of a Hermitian metric on $\mathcal{E}$ that is compatible with the flat connection. Geometrically, this compatibility ensures that the inner product of any two vectors in the fibres is strictly preserved under parallel transport along the flow. If the connection were not compatible with this Hermitian structure, the geometry of the fibres would deform along the trajectories, and this failure to preserve the inner product would give rise to an anomaly.\\

As shown in the general main theorem of \cite[Theorem 3.12]{Shen21}, for a generic representation, the equivalence between the Ray--Singer and Milnor metrics involves precisely this additional correction term, which is expressed as the integral of a characteristic class, specifically the Mathai-Quillen form associated with the connection. However, as briefly mentioned in Section \ref{sec: twisted-differential}, the exact equality without anomaly terms holds more broadly under the weaker assumption that the flat vector bundle $\mathcal{E}$ is unimodular.}\\

Finally, we note that the technical hypothesis of $C^\infty$-linearisability is not strictly necessary for the statement to hold within the framework of \cite{Shen21}. However, as will be explored in upcoming works, this linearisability condition is strictly required to treat the Fried conjecture through the lens of topological field theories, where the analysis relies on the microlocal properties of the Lie derivative and its Pollicott--Ruelle resonances. 
\nocite{*} 
\bibliographystyle{alpha} 
\bibliography{Bibliography} 

@article{Atiyah68,
  author = {Atiyah, M.F. and Bott, R.},
  title = {A {Lefschetz} fixed point formula for elliptic complexes: {II}. {Applications}},
  journal = {Annals of Mathematics},
  pages = {451--491},
  year = {1968}
}

@book{Baladi18,
  author = {Baladi, V.},
  title = {Dynamical zeta functions and dynamical determinants for hyperbolic maps},
  year = {2018},
  publisher = {Springer}
}

@article{Dang20,
  author = {Dang, N. V. and Guillarmou, C. and Rivi{\`e}re, G. and Shen, S.},
  title = {The {Fried} conjecture in small dimensions},
  journal = {Inventiones mathematicae},
  volume = {220},
  number = {2},
  pages = {525--579},
  year = {2020}
}

@article{Dyatlov16,
  author = {Dyatlov, S. and Zworski, M.},
  title = {Dynamical zeta functions for {Anosov} flows via microlocal analysis},
  journal = {Ann. Sci. Ec. Norm. Sup{\'e}r. (4)},
  volume = {49},
  number = {3},
  pages = {543--577},
  year = {2016}
}

@article{Dyatlov17,
  author = {Dyatlov, S. and Zworski, M.},
  title = {Ruelle zeta function at zero for surfaces},
  journal = {Inventiones mathematicae},
  volume = {210},
  number = {1},
  pages = {211--229},
  year = {2017}
}

@article{Fried87,
  author = {Fried, D.},
  title = {Lefschetz formulas for flows},
  journal = {Contemporary Mathematics},
  volume = {58},
  pages = {19--69},
  year = {1987}
}

@article{Fried95,
  author = {Fried, D.},
  title = {Meromorphic zeta functions for analytic flows},
  journal = {Communications in Mathematical Physics},
  volume = {174},
  number = {1},
  pages = {161--190},
  year = {1995}
}

@article{Mnev14,
  author = {Mnev, P.},
  title = {Lecture notes on torsions},
  journal = {arXiv preprint arXiv:1406.3705},
  year = {2014}
}

@article{Ray71,
  author = {Ray, D.B. and Singer, I.M.},
  title = {R-torsion and the {Laplacian} on riemannian manifolds},
  journal = {Advances in Mathematics},
  volume = {7},
  number = {2},
  pages = {145--210},
  year = {1971}
}

@article{Ruelle76,
  author = {Ruelle, D.},
  title = {Zeta-functions for expanding maps and {Anosov} flows},
  journal = {Inventiones mathematicae},
  volume = {34},
  number = {3},
  pages = {231--242},
  year = {1976}
}

@article{Ruelle86,
  author = {Ruelle, D.},
  title = {Resonances of chaotic dynamical systems},
  journal = {Physical review letters},
  volume = {56},
  number = {5},
  pages = {405},
  year = {1986}
}

@article{Shen17,
  author = {Shen, S.},
  title = {Analytic torsion, dynamical zeta functions, and the {Fried} conjecture},
  journal = {Analysis \& PDE},
  volume = {11},
  number = {1},
  pages = {1--74},
  year = {2017}
}

@article{DangRiviere17Topology,
  author        = {Dang, N. V. and Rivi{\`e}re, G.},
  title         = {Topology of {P}ollicott-{R}uelle resonant states},
  journal       = {arXiv preprint arXiv:1703.08037},
  year          = {2017}
}

@article{DangRiviere17Resonances,
  author        = {Dang, N. V. and Rivi{\`e}re, G.},
  title         = {Spectral analysis of {M}orse-{S}male flows {II}: resonances and resonant states},
  journal       = {arXiv preprint arXiv:1703.08038},
  year          = {2017}
}

@article{DangRiviere17Anisotropic,
  author        = {Dang, N. V. and Rivi{\`e}re, G.},
  title         = {Spectral analysis of {M}orse-{S}male flows {I}: construction of the anisotropic spaces},
  journal       = {arXiv preprint arXiv:1703.08040},
  year          = {2017}
}

@article{HadfieldKandelSchiavina20Ruelle,
  author        = {Hadfield, C. and Kandel, S. and Schiavina, M.},
  title         = {Ruelle zeta function from field theory},
  journal       = {arXiv preprint arXiv:2002.03952},
  year          = {2020}
}

@book{Lee12Smooth,
  author    = {Lee, J. M.},
  title     = {Introduction to Smooth Manifolds},
  series    = {Graduate Texts in Mathematics},
  volume    = {218},
  publisher = {Springer},
  year      = {2012},
  edition   = {Second},
  address   = {New York},
  isbn      = {978-1-4419-9981-8}
}

@incollection{Shen21Survey,
  author    = {Shen, S.},
  title     = {Analytic Torsion and Dynamical Flow: A Survey on the Fried Conjecture},
  booktitle = {Geometric Aspects of Functional Analysis: Israel Seminar (GAFA) 2017-2019},
  editor    = {Klartag, B. and Milman, E.},
  series    = {Progress in Mathematics},
  volume    = {338},
  pages     = {247--299},
  publisher = {Birkh{\"a}user},
  year      = {2021}
}

@book{KatokHasselblatt95,
  title={Introduction to the Modern Theory of Dynamical Systems},
  author={Katok, A. and Hasselblatt, B.},
  year={1995},
  publisher={Cambridge University Press},
  address={Cambridge},
  series={Encyclopedia of Mathematics and its Applications},
  volume={54},
  isbn={978-0-521-57557-7}
}

@book{Perko2001,
  author    = {Perko, L.},
  title     = {Differential Equations and Dynamical Systems},
  edition   = {Third},
  publisher = {Springer},
  year      = {2001},
  address   = {New York},
  isbn      = {0-387-95116-4}
}

@book{Warner1983,
  author    = {Warner, F. W.},
  title     = {Foundations of Differentiable Manifolds and Lie Groups},
  publisher = {Springer-Verlag},
  year      = {1983},
  address   = {New York},
  series    = {Graduate Texts in Mathematics},
  volume    = {94},
  isbn      = {978-1-4419-2820-7},
}

@article{MathaiWu11Twisted,
  author  = {Mathai, V. and Wu, S.},
  title   = {Analytic Torsion for Twisted de {R}ham Complexes},
  journal = {arXiv preprint arXiv:0810.4204v6},
  year    = {2011}
}

@book{Hatcher02,
  author    = {Hatcher, A.},
  title     = {{A}lgebraic {T}opology},
  publisher = {Cambridge University Press},
  year      = {2002},
  address   = {Cambridge},
  isbn      = {978-0-521-79540-1}
}

@book{Munkres00,
  author    = {Munkres, J. R.},
  title     = {{T}opology},
  publisher = {Prentice Hall},
  year      = {2000},
  edition   = {Second},
  address   = {Upper Saddle River, NJ},
  isbn      = {978-0131816299}
}

@book{Lee09ManifoldsDiffGeom,
  author    = {Lee, J. M.},
  title     = {{M}anifolds and {D}ifferential {G}eometry},
  series    = {Graduate Studies in Mathematics},
  volume    = {107},
  publisher = {American Mathematical Society},
  year      = {2009},
  address   = {Providence, RI},
  isbn      = {978-0-8218-4815-9}
}

@article{Dang21,
  author = {Dang, N. V. and Rivi{\`e}re, G.},
  title = {Pollicott-{R}uelle spectrum and {W}itten {L}aplacians},
  journal = {Journal of the European Mathematical Society},
  volume = {23},
  number = {8},
  pages = {2577--2634},
  year = {2021},
  note = {arXiv:1709.04265v1 [math.DS]}
}

@article{Chaubet22,
  author = {Chaubet, Y. and Dang, N. V.},
  title = {Dynamical torsion for contact {A}nosov flows},
  journal = {Analysis \& PDE},
  volume = {15},
  number = {8},
  pages = {1891--1959},
  year = {2022},
  note = {arXiv:1911.09931v2 [math.DS]}
}

@article{Shen21,
  author  = {Shen, S. and Yu, J.},
  title   = {Morse-{S}male flow, {M}ilnor metric, and dynamical zeta function},
  journal = {Journal de l'\'{E}cole polytechnique \textemdash{} Math\'{e}matiques},
  volume  = {8},
  pages   = {585--607},
  year    = {2021},
  note    = {arXiv:1806.00662v2 [math.DG]}
}

@techreport{Hutchings02,
  author      = {Hutchings, M.},
  title       = {Lecture notes on {M}orse homology (with an eye towards {F}loer theory and pseudoholomorphic curves)},
  institution = {UC Berkeley},
  year        = {2002},
  note        = {Lecture notes}
}

@book{Bott82,
  author = {Bott, R. and Tu, L. W.},
  title = {Differential forms in algebraic topology},
  year = {1982},
  publisher = {Springer-Verlag},
  address = {New York},
  series = {Graduate Texts in Mathematics},
  volume = {82}
}

@book{deRham84,
  author = {de Rham, G.},
  title = {Differentiable manifolds: {F}orms, currents, harmonic forms},
  year = {1984},
  publisher = {Springer-Verlag},
  address = {Berlin Heidelberg}
}

@book{Fulton91,
  author = {Fulton, W. and Harris, J.},
  title = {Representation theory: {A} first course},
  year = {1991},
  publisher = {Springer-Verlag},
  address = {New York},
  series = {Graduate Texts in Mathematics},
  volume = {129}
}

@book{Palis82,
  author = {Palis, J. and de Melo, W.},
  title = {Geometric theory of dynamical systems: {A}n introduction},
  year = {1982},
  publisher = {Springer-Verlag},
  address = {New York}
}

@book{Hasselblatt03,
  author = {Hasselblatt, B. and Katok, A.},
  title = {A first course in dynamics: {W}ith a panorama of recent developments},
  year = {2003},
  publisher = {Cambridge University Press},
  address = {Cambridge}
}

@book{Marsden88,
  author = {Marsden, J. E. and Ratiu, T. S.},
  title = {Manifolds, tensor analysis, and applications},
  year = {1988},
  publisher = {Springer-Verlag},
  address = {New York},
  series = {Applied Mathematical Sciences},
  volume = {75},
  edition = {Second}
}

@article{Hutchings02Reidemeister,
  author  = {Hutchings, M.},
  title   = {Reidemeister torsion in generalized {M}orse theory},
  journal = {Forum Mathematicum},
  volume  = {14},
  number  = {2},
  pages   = {209--244},
  year    = {2002},
  note    = {arXiv:math/9907066v2 [math.DG]}
}

@article{Hutchings99Circle,
  author  = {Hutchings, M. and Lee, Y.-J.},
  title   = {Circle-valued {M}orse theory, {R}eidemeister torsion, and {S}eiberg-{W}itten invariants of 3-manifolds},
  journal = {Topology},
  volume  = {38},
  number  = {4},
  pages   = {861--888},
  year    = {1999},
  note    = {arXiv:dg-ga/9612004v1}
}

@article{Bismut94,
  author = {Bismut, J.-M. and Zhang, W.},
  title = {Milnor and {R}ay-{S}inger metrics on the equivariant determinant of a flat vector bundle},
  journal = {Geometric and Functional Analysis},
  volume = {4},
  number = {2},
  pages = {136--212},
  year = {1994}
}

@book{Bismut92,
  author = {Bismut, J.-M. and Zhang, W.},
  title = {An extension of a theorem by {C}heeger and {M}{\"u}ller},
  series = {Ast{\'e}risque},
  volume = {205},
  year = {1992},
  publisher = {Soci{\'e}t{\'e} Math{\'e}matique de France},
  address = {Paris},
  note = {With an appendix by Fran{\c c}ois Laudenbach}
}

@incollection{Lau92,
  title={Appendix: On the {Thom-Smale} complex},
  author={Laudenbach, Fran{\c{c}}ois},
  booktitle={An extension of a theorem by {Cheeger} and {M\"uller}},
  journal={Ast{\'e}risque},
  volume={205},
  pages={219--233},
  year={1992},
  publisher={Soci{\'e}t{\'e} Math{\'e}matique de France}
}

@book{Turaev01,
  author    = {Turaev, Vladimir},
  title     = {Introduction to combinatorial torsions},
  publisher = {Birkh{\"a}user Verlag},
  year      = {2001},
  address   = {Basel},
  series    = {Lectures in Mathematics ETH Z{\"u}rich},
  note      = {Notes taken by Felix Schlenk}
}

@misc{Nicolaescu03,
  author = {Nicolaescu, Liviu I.},
  title = {Notes on the {Reidemeister} Torsion},
  year = {2003},
  note = {Lecture notes for the course MATH 927, University of Notre Dame, Fall 2002}
}

@article{Bouwknegt02,
  author = {Bouwknegt, P. and Carey, A. and Mathai, V. and Murray, M. and Stevenson, D.},
  title = {Twisted K-theory and K-theory of bundle gerbes},
  journal = {Communications in Mathematical Physics},
  volume = {228},
  number = {1},
  pages = {17--49},
  year = {2002},
  archivePrefix = {arXiv},
  eprint = {hep-th/0106194}
}

@article{Smale67,
  author = {Smale, S.},
  title = {Differentiable dynamical systems},
  journal = {Bulletin of the American Mathematical Society},
  volume = {73},
  number = {6},
  pages = {747--817},
  year = {1967},
  publisher = {American Mathematical Society}
}

@inproceedings{Seeley67,
  author    = {Seeley, R. T.},
  title     = {Complex powers of an elliptic operator},
  booktitle = {Singular Integrals (Proc. Sympos. Pure Math., Chicago, Ill., 1966)},
  volume    = {10},
  pages     = {288--307},
  publisher = {Amer. Math. Soc.},
  address   = {Providence, R.I.},
  year      = {1967}
}

@article{Cheeger79,
  author  = {Cheeger, Jeff},
  title   = {Analytic torsion and the heat equation},
  journal = {Annals of Mathematics},
  volume  = {109},
  number  = {2},
  pages   = {259--322},
  year    = {1979}
}

@article{Muller78,
  author  = {M{\"u}ller, Werner},
  title   = {Analytic torsion and {R}-torsion of {R}iemannian manifolds},
  journal = {Advances in Mathematics},
  volume  = {28},
  number  = {3},
  pages   = {233--305},
  year    = {1978}
}

@book{Franks82,
  author    = {Franks, John M.},
  title     = {Homology and dynamical systems},
  series    = {CBMS Regional Conference Series in Mathematics},
  volume    = {49},
  publisher = {American Mathematical Society},
  year      = {1982},
  address   = {Providence, R.I.},
  note      = {Published for the Conference Board of the Mathematical Sciences, Washington, D.C.},
  mrnumber  = {669378}
}

@article{FaureSjostrand11,
  author  = {Faure, F. and Sj{\"o}strand, J.},
  title   = {{U}pper bound on the density of {R}uelle resonances for {A}nosov flows},
  journal = {Communications in Mathematical Physics},
  volume  = {308},
  number  = {2},
  pages   = {325--364},
  year    = {2011}
}

@article{Reidemeister35,
  author = {Reidemeister, K.},
  title = {Homotopieringe und {L}insenr{\"a}ume},
  journal = {Abhandlungen aus dem Mathematischen Seminar der Universit{\"a}t Hamburg},
  volume = {11},
  pages = {102--109},
  year = {1935}
}

@article{Franz35,
  author = {Franz, W.},
  title = {{\"U}ber die {T}orsion einer {{\"U}}berdeckung},
  journal = {Journal f{\"u}r die reine und angewandte Mathematik},
  volume = {173},
  pages = {245--254},
  year = {1935}
}

@article{RaySinger71,
  author = {Ray, D. B. and Singer, I. M.},
  title = {{$R$}-torsion and the {L}aplacian on {R}iemannian manifolds},
  journal = {Advances in Mathematics},
  volume = {7},
  number = {2},
  pages = {145--210},
  year = {1971}
}

@book{Milnor63,
  author = {Milnor, J.},
  title = {Morse Theory},
  series = {Annals of Mathematics Studies},
  volume = {51},
  publisher = {Princeton University Press},
  address = {Princeton, N.J.},
  year = {1963}
}

@article{Smale61,
  author = {Smale, S.},
  title = {On gradient dynamical systems},
  journal = {Annals of Mathematics},
  volume = {74},
  number = {1},
  pages = {199--206},
  year = {1961}
}

@article{SM96,
  title={Reidemeister torsion and {Morse-Smale} flows},
  author={S{\'a}nchez-Morgado, H{\'e}ctor},
  journal={Ergodic Theory and Dynamical Systems},
  volume={16},
  number={2},
  pages={405--414},
  year={1996},
  publisher={Cambridge University Press}
}
\end{document}